\documentclass[
	a4paper,					
	english,
]{article} 

\usepackage[hyphens]{url} 	
\usepackage{geometry} 		
\usepackage[T1]{fontenc}		
\usepackage[utf8]{inputenc} 	
\usepackage{lmodern}			
\usepackage[english]{babel} 	
\usepackage{listings} 		
\usepackage{graphicx}		
\usepackage[]{algorithm2e} 	
\usepackage{mathrsfs}		

\usepackage{pgf,tikz}		
\usetikzlibrary{arrows}
\usetikzlibrary{arrows.meta} 
\usepackage{pgfplots} 		
\usepackage{pgfplotstable} 	
\pgfplotsset{compat=newest}
\usepackage{booktabs}
\pgfplotstableset{
every head row/.style={before row=\toprule,after row=\midrule},
every last row/.style={after row=\bottomrule}}

\usepackage{gensymb}			
\usepackage{array} 			
\usepackage{enumerate} 		
\usepackage{textcomp}		
\usepackage{subcaption} 	
\newcommand\newsubcap[1]{\phantomcaption%
       \caption*{\figurename~\thefigure\thesubfigure: #1}}
\usepackage{xcolor}			
\usepackage{float} 			
\usepackage{scrdate} 		
\usepackage{tabularx} 		
\usepackage{setspace} 		

\usepackage{mathtools}		
\usepackage{amssymb}			
\usepackage{stmaryrd}		
\SetSymbolFont{stmry}{bold}{U}{stmry}{m}{n} 	
\usepackage[thmmarks, amsmath, amsthm]{ntheorem} 	
\usepackage{bm}				
\usepackage{kbordermatrix} 	
\usepackage{gauss} 			
\usepackage{blkarray} 		
\usepackage{nicematrix} 		
\usepackage{upgreek} 		
\usepackage[hidelinks, pdfstartview = FitB]{hyperref} 
\usepackage[nameinlink]{cleveref} 

\usepackage{nicefrac}		

\usepackage{titling} 		
\usepackage{lipsum} 			

\usepackage{fancyhdr} 		
\usepackage{titlesec} 		
\titleformat*{\section}{\large\bfseries} 
\usepackage{enumitem} 		
\usepackage[hypcap=true]{caption}

\usepackage{csquotes}		

\newtheorem{lemma}{Lemma} 
\numberwithin{lemma}{section} 
\newtheorem{definition}[lemma]{Definition} 
\newtheorem{theorem}[lemma]{Theorem} 
 
\newtheorem{corollary}[lemma]{Corollary} 
\newtheorem{remark}[lemma]{Remark} 

\DeclareMathOperator*{\argmin}{\arg\!\min} 
\DeclareMathOperator{\spn}{span} 
\DeclareMathOperator{\Harm}{Harm}

\newcommand{\vc}[1]{\bm{#1}} 
\newcommand{\mt}[1]{\bm{#1}} 

\newcommand{\SO}[1]{\mathcal{SO}(#1)} 
\newcommand{\inv}[1]{{#1}^{-1}} 
\newcommand{\mapping}[3]{{#1}\colon{#2}\to{#3}} 
\newcommand{\concat}[2]{{#1}\circ{}{#2}} 
\newcommand{\concatt}[3]{{#1}\circ{}{#2}\circ{}{#3}} 

\newcolumntype{C}[1]{>{\centering\arraybackslash}p{#1}}

\makeatletter
\let\@@pmod\pmod
\DeclareRobustCommand{\pmod}{\@ifstar\@pmods\@@pmod}
\def\@pmods#1{\mkern4mu({\operator@font mod}\mkern 6mu#1)}
\makeatother

\fussy						
\graphicspath{{Bilder/} {Design/}} 

\usepackage[final]{microtype}	
\usepackage{scrhack}			

\usepackage{orcidlink}

\newif\ifmarkedits

\markeditsfalse

\newcommand{\edited}[1]{%
  \ifmarkedits
    {\color{red}#1}%
  \else
    #1%
  \fi
}

\begin{document}
\setcounter{section}{0} 
\setcounter{subsection}{1}
\include{design/metadata}	
\newcommand{\sA}{{\mathcal A}}
\newcommand{\sB}{{\mathcal B}}
\newcommand{\sC}{{\mathcal C}}
\newcommand{\sD}{{\mathcal D}}
\newcommand{\sE}{{\mathcal E}}
\newcommand{\sF}{{\mathcal F}}
\newcommand{\sG}{{\mathcal G}}
\newcommand{\sH}{{\mathcal H}}
\newcommand{\sI}{{\mathcal I}}
\newcommand{\sJ}{{\mathcal J}}
\newcommand{\sK}{{\mathcal K}}
\newcommand{\sL}{{\mathcal L}}
\newcommand{\sM}{{\mathcal M}}
\newcommand{\sN}{{\mathcal N}}
\newcommand{\sO}{{\mathcal O}}
\newcommand{\sP}{{\mathcal P}}
\newcommand{\sQ}{{\mathcal Q}}
\newcommand{\sR}{{\mathcal R}}
\newcommand{\sS}{{\mathcal S}}
\newcommand{\sT}{{\mathcal T}}
\newcommand{\sU}{{\mathcal U}}
\newcommand{\sV}{{\mathcal V}}
\newcommand{\sW}{{\mathcal W}}
\newcommand{\sX}{{\mathcal X}}
\newcommand{\sY}{{\mathcal Y}}
\newcommand{\sZ}{{\mathcal Z}}

\newcommand{\scrA}{{\mathscr A}}
\newcommand{\scrB}{{\mathscr B}}
\newcommand{\scrC}{{\mathscr C}}
\newcommand{\scrD}{{\mathscr D}}
\newcommand{\scrE}{{\mathscr E}}
\newcommand{\scrF}{{\mathscr F}}
\newcommand{\scrG}{{\mathscr G}}
\newcommand{\scrH}{{\mathscr H}}
\newcommand{\scrI}{{\mathscr I}}
\newcommand{\scrJ}{{\mathscr J}}
\newcommand{\scrK}{{\mathscr K}}
\newcommand{\scrL}{{\mathscr L}}
\newcommand{\scrM}{{\mathscr M}}
\newcommand{\scrN}{{\mathscr N}}
\newcommand{\scrO}{{\mathscr O}}
\newcommand{\scrP}{{\mathscr P}}
\newcommand{\scrQ}{{\mathscr Q}}
\newcommand{\scrR}{{\mathscr R}}
\newcommand{\scrS}{{\mathscr S}}
\newcommand{\scrT}{{\mathscr T}}
\newcommand{\scrU}{{\mathscr U}}
\newcommand{\scrV}{{\mathscr V}}
\newcommand{\scrW}{{\mathscr W}}
\newcommand{\scrX}{{\mathscr X}}
\newcommand{\scrY}{{\mathscr Y}}
\newcommand{\scrZ}{{\mathscr Z}}

\newcommand{\A}{{\mathbb A}}
\newcommand{\B}{{\mathbb B}}
\newcommand{\C}{{\mathbb C}}	
\newcommand{\D}{{\mathbb D}}
\newcommand{\E}{{\mathbb E}}
\newcommand{\F}{{\mathbb F}}
\newcommand{\G}{{\mathbb G}}	
\newcommand{\HH}{{\mathbb H}}
\newcommand{\I}{{\mathbb I}}
\newcommand{\J}{{\mathbb J}}
\renewcommand{\L}{{\mathbb L}}
\newcommand{\M}{{\mathbb M}}
\newcommand{\K}{{\mathbb K}}
\newcommand{\N}{{\mathbb N}}
\renewcommand{\P}{{\mathbb P}}
\newcommand{\Q}{{\mathbb Q}}
\newcommand{\R}{{\mathbb R}}
\renewcommand{\S}{{\mathbb S}}
\newcommand{\T}{{\mathbb T}}
\newcommand{\U}{{\mathbb U}} 	
\newcommand{\V}{{\mathbb V}}
\newcommand{\W}{{\mathbb W}}
\newcommand{\X}{{\mathbb X}}
\newcommand{\Y}{{\mathbb Y}}
\newcommand{\Z}{{\mathbb Z}}

\newcommand{\liee}{\mathfrak{e}}
\newcommand{\lief}{\mathfrak{f}}
\newcommand{\lieg}{\mathfrak{g}}
\newcommand{\lieh}{\mathfrak{h}}
\newcommand{\liei}{\mathfrak{i}}
\newcommand{\liej}{\mathfrak{j}}
\newcommand{\liek}{\mathfrak{k}}
\newcommand{\liel}{\mathfrak{l}}
\newcommand{\liem}{\mathfrak{m}}
\newcommand{\lien}{\mathfrak{n}}
\newcommand{\liegl}{\mathfrak{gl}}
\newcommand{\liesl}{\mathfrak{sl}}
\newcommand{\liesp}{\mathfrak{sp}}
\newcommand{\lieso}{\mathfrak{so}}
\newcommand{\liesu}{\mathfrak{su}}
\newcommand{\lies}{\mathfrak{s}}
\newcommand{\liet}{\mathfrak{t}}
\newcommand{\lieA}{\mathfrak{A}}
\newcommand{\gotha}{{\mathfrak a}}
\newcommand{\gothb}{{\mathfrak b}}
\newcommand{\gothc}{{\mathfrak c}}
\newcommand{\gothd}{{\mathfrak d}}
\newcommand{\gothe}{{\mathfrak e}}
\newcommand{\gothf}{{\mathfrak f}}
\newcommand{\gothg}{{\mathfrak g}}
\newcommand{\gothh}{{\mathfrak h}}
\newcommand{\gothi}{{\mathfrak i}}
\newcommand{\gothj}{{\mathfrak j}}
\newcommand{\gothk}{{\mathfrak k}}
\newcommand{\gothl}{{\mathfrak l}}
\newcommand{\gothm}{{\mathfrak m}}
\newcommand{\gothn}{{\mathfrak n}}
\newcommand{\gotho}{{\mathfrak o}}
\newcommand{\gothp}{{\mathfrak p}}
\newcommand{\gothq}{{\mathfrak q}}
\newcommand{\gothr}{{\mathfrak r}}
\newcommand{\goths}{{\mathfrak s}}
\newcommand{\gotht}{{\mathfrak t}}
\newcommand{\gothu}{{\mathfrak u}}
\newcommand{\gothv}{{\mathfrak v}}
\newcommand{\gothw}{{\mathfrak w}}
\newcommand{\gothx}{{\mathfrak x}}
\newcommand{\gothy}{{\mathfrak y}}
\newcommand{\gothz}{{\mathfrak z}}
\newcommand{\gothA}{{\mathfrak A}}
\newcommand{\gothB}{{\mathfrak B}}
\newcommand{\gothC}{{\mathfrak C}}
\newcommand{\gothD}{{\mathfrak D}}
\newcommand{\gothE}{{\mathfrak E}}
\newcommand{\gothF}{{\mathfrak F}}
\newcommand{\gothG}{{\mathfrak G}}
\newcommand{\gothH}{{\mathfrak H}}
\newcommand{\gothI}{{\mathfrak I}}
\newcommand{\gothJ}{{\mathfrak J}}
\newcommand{\gothK}{{\mathfrak K}}
\newcommand{\gothL}{{\mathfrak L}}
\newcommand{\gothM}{{\mathfrak M}}
\newcommand{\gothN}{{\mathfrak N}}
\newcommand{\gothO}{{\mathfrak O}}
\newcommand{\gothP}{{\mathfrak P}}
\newcommand{\gothQ}{{\mathfrak Q}}
\newcommand{\gothR}{{\mathfrak R}}
\newcommand{\gothS}{{\mathfrak S}}
\newcommand{\gothT}{{\mathfrak T}}
\newcommand{\gothU}{{\mathfrak U}}
\newcommand{\gothV}{{\mathfrak V}}
\newcommand{\gothW}{{\mathfrak W}}
\newcommand{\gothX}{{\mathfrak X}}
\newcommand{\gothY}{{\mathfrak Y}}
\newcommand{\gothZ}{{\mathfrak Z}}


\newcommand{\abs}[1]{{\left|#1\right|}}
\newcommand{\acts}{\mathrel{\raisebox{0.3ex}{\rotatebox[origin=c]{270}{$\circlearrowright$}}}}
\newcommand{\ab}{{\operatorname{ab}}}
\newcommand{\an}{{\operatorname{an}}}
\newcommand{\Ann}{\operatorname{Ann}}
\newcommand{\Art}{\operatorname{Art}}
\newcommand{\bild}{{\operatorname{bild}}}
\newcommand{\proArt}{{\widehat{\operatorname{Art}}}}
\newcommand{\Aut}{\operatorname{Aut}}
\newcommand{\ba}[1]{\overline{#1}}
\newcommand{\id}{{\rm id}}
\newcommand{\ch}{{\rm ch}}
\newcommand{\car}{\operatorname{char}}
\newcommand{\chern}{{\rm c}}
\newcommand{\codim}{\operatorname{codim}}
\newcommand{\Coh}{\operatorname{Coh}}
\newcommand{\coker}{\operatorname{coker}}
\newcommand{\Der}{\operatorname{Der}}
\newcommand{\di}{\partial}
\newcommand{\End}{\operatorname{End}}
\newcommand{\Ext}{\operatorname{Ext}}
\newcommand{\Gal}{\operatorname{Gal}}
\newcommand{\eins}[1]{{{\mathbf 1}_{#1 \times #1}}}
\newcommand{\eps}{\varepsilon}
\newcommand{\equivalent}{\Leftrightarrow}
\newcommand{\ev}{\operatorname{ev}}
\newcommand{\GL}{{\rm GL}}
\newcommand{\gm}{\nabla^{\rm GM}}
\newcommand{\Gr}{{\rm Gr}}
\newcommand{\Hilb}{{\rm Hilb}}
\newcommand{\Hom}{\operatorname{Hom}}
\newcommand{\hinrichtung}{``$\Longrightarrow$''}
\newcommand{\rueckrichtung}{``$\Longleftarrow$''}
\newcommand{\Iso}{\operatorname{Iso}}
\newcommand{\iso}{\cong}
\renewcommand{\implies}{{\ \Rightarrow\ }}
\renewcommand{\iff}{\Leftrightarrow}
\newcommand{\gdw}{\Longleftrightarrow}
\renewcommand{\Im}{\operatorname{Im}}
\newcommand{\img}{\operatorname{im}}
\newcommand{\into}{{\, \hookrightarrow\,}}
\newcommand{\isom}{\cong}
\newcommand{\limes}[1][]{{\varprojlim_{ #1 }\;}}
\newcommand{\kolimes}[1][]{{\varinjlim_{ #1 }\;}}
\newcommand{\kgv}{\operatorname{kgV}}
\newcommand{\kong}{{\ \equiv \ }}
\newcommand{\length}{\operatorname{length}}
\newcommand{\ggt}{{\rm{ggT}}}
\newcommand{\grad}{{\operatorname{grad}}}
\newcommand{\margincom}[1]{\marginpar{\sffamily\tiny #1}}
\newcommand{\Mor}{\operatorname{Mor}}
\newcommand{\nil}{{\operatorname{nil}}}
\newcommand{\nt}{{\vartriangleleft\ }} 
\newcommand{\normal}{{\vartriangleleft\ }} 
\newcommand{\NS}{\operatorname{NS}}
\newcommand{\Num}{\operatorname{Num}}
\renewcommand{\O}{{\rm O}}
\newcommand{\ob}{{\operatorname{ob}}}
\newcommand{\Ob}{\operatorname{Ob}}
\newcommand{\ohne}{{\ \setminus \ }}
\newcommand{\ohnenull}{{\ \setminus \{0\} }}
\newcommand{\ol}[1]{{\overline{#1}}}
\newcommand{\onto}{{{\twoheadrightarrow}}}
\newcommand{\op}{{\rm op}}
\newcommand{\ord}{{\rm ord}}
\newcommand{\para}{{$\mathsection\,$}}
\newcommand{\Psh}{{\rm Psh}}
\newcommand{\Pic}{\operatorname{Pic}}
\newcommand{\pr}{\operatorname{pr}}
\newcommand{\qis}{\simeq}
\newcommand{\ran}{\operatorname{ran}}
\newcommand{\rank}{{\rm rank}}
\newcommand{\RHom}{\operatorname{RHom}}
\renewcommand{\Re}{\operatorname{Re}}
\newcommand{\ratl}{\dashrightarrow}
\newcommand{\red}{{\operatorname{red}}}
\newcommand{\reg}{{\operatorname{reg}}}
\newcommand{\rg}{\operatorname{rg}}
\newcommand{\rk}{{\rm rk}}
\newcommand{\sing}{{\operatorname{sing}}}
\newcommand{\Sing}{\operatorname{Sing}}
\newcommand{\sgn}{sgn}
\newcommand{\SL}{{\rm SL}}
\newcommand{\Sh}{{\rm Sh}}
\newcommand{\socle}{{\rm socle}}
\newcommand{\Schm}{{\rm Sch }}
\newcommand{\Sch}[1]{{\rm Sch / #1}}
\newcommand{\Set}{{\rm Set}}
\newcommand{\sep}{{{\rm sep}}}
\newcommand{\Sp}{{\rm Sp}}
\newcommand{\Spec}{{\operatorname{Spec}}}
\newcommand{\Spm}{{\operatorname{Spm}}}
\newcommand{\Spin}{{\rm Spin}}
\newcommand{\st}{{\operatorname{1st}}}
\newcommand{\SU}{{\rm SU}}
\newcommand{\supp}{\operatorname{supp}}
\newcommand{\Sym}{{\rm Sym}}
\renewcommand{\to}[1][]{\xrightarrow{\ #1\ }}
\newcommand{\Tor}{{\rm Tor}}
\newcommand{\Tors}{{\rm Tors}}
\newcommand{\tensor}{\otimes}
\newcommand{\tOm}{\widetilde{\Omega}}
\newcommand{\tr}{{\rm tr}}
\newcommand{\td}{{\rm td}}
\newcommand{\tY}{{\widetilde{Y}}}
\newcommand{\tsY}{{\widetilde{\sY}}}
\newcommand{\veps}{\varepsilon}
\newcommand{\vol}{\operatorname{vol}}
\newcommand{\vphi}{\varphi}
\newcommand{\vrho}{{\varrho}}
\newcommand{\vthe}{\vartheta}
\renewcommand{\U}{{\rm U}}
\newcommand{\ul}[1]{{\underline{#1}}}
\newcommand{\wh}[1]{{\widehat{#1}}}
\newcommand{\wlfun}[1]{{\underline{#1_{\wl}}}}
\newcommand{\wt}[1]{{\widetilde{#1}}}
\newcommand{\Zar}{{\rm Zar}}



\title{An Optimal Ansatz Space for \\Moving Least Squares Approximation on
  Spheres}

\author{Ralf
  Hielscher\orcidlink{0000-0002-6342-1799}\thanks{ralf.hielscher@math.tu-freiberg.de}~,~
  Tim Pöschl\orcidlink{0009-0006-8313-9850}\thanks{tim.poeschl@math.tu-freiberg.de}}

\date{Institute for Applied Analysis, TU Bergakademie Freiberg \\
  September 30, 2024}

\maketitle
\begin{abstract}
  We revisit the moving least squares (MLS)
  approximation scheme on the sphere $\mathbb S^{d-1} \subset \R^d$, where
  $d>1$. It is well known that using the spherical harmonics up to degree
  $L \in \N$ as ansatz space yields for functions in
  $\sC^{L+1}(\mathbb S^{d-1})$ the approximation order
  $\sO \left( h^{L+1} \right)$, where $h$ denotes the fill distance of the
  sampling nodes.

  In this paper we show that the dimension of the ansatz space can be almost
  halved, by including only spherical harmonics of even or odd degree up to $L$,
  while preserving the same order of approximation. Numerical experiments
  indicate that using the reduced ansatz space is essential to ensure the
  numerical stability of the MLS approximation scheme as $h \to 0$.  Finally, we
  compare our approach with an MLS approximation scheme that uses polynomials on
  the tangent space of the sphere as ansatz space.
\end{abstract}


\ \newline\noindent\textbf{Acknowledgements}
This version of the article has been accepted for publication, after peer
review, but is not the Version of Record and does not reflect
post-acceptance improvements, or any corrections. The Version of Record is
available online at: \url{http://dx.doi.org/10.1007/s10444-024-10201-z}.

\section{Introduction}

Spherical data appear in many different scientific applications, such as
geography and geodesy, astronomy, computer graphics and even in biology in
protein docking simulation. This explains the need for fast and stable
algorithms for the reconstruction of spherical functions from discrete data.

One approach is to use spherical harmonics up to a fixed degree as a global
ansatz space and to determine the coefficients by minimizing some error
functional, see e.g.\ \cite{AtHa12}. The choice of the polynomial degree as
well as of the error functional are crucial to ensure the stability of the
approximation process \cite{FiTh06,Ku07,filbir2023marcinkiewiczzygmund}.  A
second issue is the efficient numerical solution of the minimization problem
which has been tackled, e.g.\ with the fast nonequispaced spherical Fourier
transform \cite{potts.1998, potts.2008} or the double Fourier sphere method
\cite{doublesphere, doulesphereNdim}. A completely different approach are
local approximation schemes, where many local approximations are computed
instead of one global one. This avoids the costly solution of a large system
of equations and stability can be resolved locally. Radial basis function
approximation is a method that can be applied in a global scheme
\cite{Golitschek.2001,KurtJetter.1999} as well as in a local one
\cite{rbflocal}.

In this paper we revisit the moving least squares (MLS) approximation scheme
\cite{Levin.1998,Sober.,Wendland.2001b,Wendland.2001} on the Euclidean sphere
which is a purely local method. Assume that $f \colon \S^{d-1} \to \C$ is a
function from the $d-1$ dimensional sphere
$\S^{d-1} = \{y \in \R^{d} \mid \lVert y \rVert_{2} = 1\}$ to the complex numbers
with known values at the nodes $y_{1},\dots,y_{N} \in \S^{d-1}$. Let
$w \colon \S^{d-1} \times \S^{d-1} \to \R_{\geq 0}$ denote a weight function
and let $\sG \subset \sC \left( \S^{d-1} \right)$ be a finite dimensional,
linear ansatz space over $\C$. Under mild assumptions, cf.\ \cite{Levin.1998},
there exists a unique function $g_y \in \sG$ solving the weighted least
squares problem
\begin{align}
  \label{eq:mls standard formulation}
  \min\limits_{g \in \sG} \sum\limits_{i=1}^N w(y_i,y) \vert f(y_i) - g(y_i) \vert^2.
\end{align}
The ansatz space $\sG$ is usually chosen very low dimensional, which allows
the weight function $w(\cdot,y)$ to be localized around $y \in \S^{d-1}$ with only
few nodes $y_{i}$ close to the center $y$ satisfying $w(y_{i},y)>0$. As a result $g_{y}$ is
only a local approximation to $f$ that is very fast to compute. In order to
guarantee the stable solution of \eqref{eq:mls standard formulation}, the weight
function needs to be carefully adapted to the nodes $y_{i}$ and the function
space $\sG$.

The MLS approximation $\sM f(y)$ of $f$ is obtained by solving the
local minimization problem \eqref{eq:mls standard formulation} separately for
each evaluation point $y \in \S^{d-1}$ and setting $\sM f(y) \coloneqq g_y(y)$. An
overview on the properties of the MLS approximation can be found in
\cite[chapter 22-25]{Fasshauer.2008} and in \cite{Levin.1998}.

In our paper we consider the asymptotic behavior of the approximation error
\begin{equation}
  \label{eq:sup norm}
  \| f - \sM f \|_\infty
  = \sup\limits_{y \in \S^{d-1}} \vert f(y) - \sM f(y) \vert
\end{equation}
of the MLS approximation as the number $N$ of nodes increases, which
simultaneously allows the weight function $w$ to become more and more localized.
In order to bound the approximation error \eqref{eq:sup norm}, the nodes $y_{i}$
need to be distributed sufficiently uniformly in $\S^{d-1}$. The uniformity
of the nodes can be quantified by the ratio of their separation distance $q$ and
their fill distance $h$, see \Cref{def:grid properties}, which refer to the
smallest distance between two nodes and the radius of the biggest hole,
respectively. For $\Omega \subset \R^d$ it is well known, cf.\ \cite{Levin.1998,
  Wendland.2001b}, that for $f \in C^{L+1}(\Omega)$ and $h \to 0$ we
have $\| f - \sM f \|_\infty = \sO \left( h^{L+1} \right)$, if the uniformity
$h/q$ is bounded by some constant $c$ and $\sG$ consists of all polynomials up
to degree $L$.

A similar result for $\S^{d-1}$ was shown by Wendland \cite{Wendland.2001}.
In his work he proved the asymptotic bound
$\| f - \sM f \|_\infty = \sO \left( h^{L+1} \right)$ whenever
$f \in C^{L+1}(\S^{d-1})$ and $\sG$ consists of all spherical harmonics up to
degree $L$.

In our paper we are going to show, that the dimension of the ansatz space
$\sG$ can be almost halved while preserving the same order of approximation.
To be precise, if $L$ is even it will be sufficient to only use the spherical
harmonics of even degree up to $L$, and if $L$ is odd it will be sufficient to
only use the spherical harmonics of odd degree up to $L$, in order to obtain
approximation order $\sO \left( h^{L+1} \right)$ for $h \to 0$. This is
possible, because the spherical harmonics up to degree $L$ are linearly
independent from a global point of view, but they asymptotically become linearly
dependent on a local scale. This is not the case anymore if we use only spherical
harmonics of every second degree. Therefore, the smaller ansatz space also increases
the numerical stability.

An alternative idea to keep the ansatz space $\sG$ well suited for approximation
problems that are more and more localized around some center $y$, is to allow it
to vary with $y$ as well, i.e.\ to solve \eqref{eq:mls standard formulation} with
respect to a local ansatz space $\sG_{y}$. A natural choice for such an ansatz
space is the linear space consisting of the pullback functions of polynomials
on the tangent space to the sphere. This approach was
already discussed for general manifolds in \cite{Sober.}. We show, that this
method possesses the same approximation order $\sO \left( h^{L+1} \right)$ for
$h \to 0$ as our global ansatz space.

Our findings are organized as follows. In 
\cref{sec:MLS} we give a quick introduction into moving least
squares approximation. The essential point is that the approximation error can
be bounded by the product of some Lebesgue constant and the error of the local
best approximation, cf.\ \Cref{lem:MLS error estimate}. Local
approximation on the sphere is discussed in \cref{sec:localApprox}. In
particular, we prove in \Cref{lem:Phi(G) = all polynomials up to degree
  L} conditions on the ansatz space $\mathcal G$, that guarantee that the error
of the local best approximation of a function $f \in C^{L+1}(\S^{d-1})$ within
a $\delta$-ball around $y$ decays as $\mathcal O(\delta^{L+1})$. Subsequently,
we show that these conditions are satisfied by the spaces of spherical harmonics 
of even, respectively odd, degree up to $L$ and summarize in \Cref{thm:main result}
the local approximation properties of both spaces.

\Cref{sec:lebesgue-constant} is devoted to the computation of the
Lebesgue constant for the spaces of spherical harmonics of even, respectively odd,
degree up to $L$ generalizing a result of Wendland, cf.\ \cite{Wendland.2001}.
In \Cref{thm:lebesgue constant bounded} we give conditions on the separation 
distance and on the fill distance of the nodes $y_{i}$, as well as on
the localization of the weight function $w$, such that the Lebesgue constant is
uniformly bounded. Combining all these ingredients we end up with our final
approximation result in \Cref{cor:approx order of MLS}.

\Cref{sec:tangentSpace} considers the case of a local ansatz
space $\mathcal G_{y}$, which consists of the pullbacks functions of polynomials
on the tangent space to the sphere. The two important factors, namely the error
of the local best approximation and the Lebesgue constant, are estimated in
\Cref{lem:tangent local approx order L+1} and \Cref{lem:tangent lebesgue
  constant bounded}, respectively. The final approximation result can be found in
\Cref{cor:tanget_space}.

Eventually, in \cref{sec:numerics}, we numerically compare the three
choices of the ansatz space: all spherical harmonics up to degree $L$, only the
even, respectively odd, spherical harmonics up to degree $L$, and all polynomials up to degree $L$ on the
tangent space of the sphere. Our numerical experiments indicate, that while all three approaches attain 
the approximation error proven in this paper, the latter two approaches result in much more stable
algorithms, especially if the basis is chosen carefully.

 
\section{Moving Least Squares Approximation}
\label{sec:MLS}

We start by giving an explicit definition of moving least squares approximation
which dates back to Shepard \cite{shepard68}.

\begin{definition}
  \label{def:def mls}
  Let a function $f \colon \S^{d-1} \to \C$ and a center $y \in \S^{d-1}$ be
  given. Further, let $y_{1}, \dots, y_{N} \in \S^{d-1}$ be nodes,
  $\sG = \spn \left\{ g_{1},\dots,g_{M} \colon \S^{d-1} \to \R \right\}$
  a linear ansatz space over $\C$ and
  $w \colon \S^{d-1} \times \S^{d-1} \to \R_{\geq 0}$ a weight function, such
  that the minimization problem
  \begin{equation}
    \label{eq:def mls}
    \min\limits_{g \in \sG} \, \sum\limits_{i=1}^{N}w(y,y_{i}) \, \vert  f(y_{i}) - g(y_{i}) \vert^{2}
  \end{equation}
  possesses a unique solution $g_{y} \in \sG$. Then
  $\sM f(y) \coloneqq g_{y} (y)$ is the MLS approximation of $f$ in $y$.
\end{definition}

\begin{remark}
  Of course the ansatz space $\sG$ can be chosen as a linear space over $\R$
  instead of $\C$ if $f$ is real-valued. Throughout the paper we will not
  distinguish between both cases. The ansatz space and all polynomial spaces
  defined throughout this paper should always be seen as linear spaces over the
  same field that $f$ maps into.
\end{remark}

It is easy to check that the solution of \eqref{eq:def mls} exists and is
unique, if and only if the nodes
$\left\{ y_{i} \, \big\vert \, w(y_{i},y) > 0\right\}$ are unisolvent with
respect to $\sG$, that is $g \equiv 0$ is the only function in $\sG$ that
vanishes on all these nodes. In this case, the MLS approximation coincides
with the Backus-Gilbert approximation \cite{backus67, bos89:_movin}. For a
proof of the following theorem see \cite[Prop. 1]{Levin.1998}.

\begin{theorem}[Backus-Gilbert Approximation]
  \label{thm:mls and backus gilbert}
  The MLS approximation can be written as
  \begin{equation}
    \label{eq:mls backus gilbert}
    \sM f(y) = \sum_{w(y_i,y) > 0} a_i^{\ast}(y) f(y_i) ,
  \end{equation}
  where the coefficients $a_i^\ast (y) \in \C$ are the unique solution of the
  minimization problem
  \begin{align}
    &\min\limits_{a_i(y) \in \C} \sum\limits_{w(y_i,y) > 0} \frac{1}{w(y_i,y)} \vert a_i(y) \vert^2
      \label{eq:backus gilbert target function} \\
    \text{subject to} \quad
    &\sum\limits_{w(y_i,y) > 0} a_i(y) \, g(y_i) = g(y) , \quad g \in \sG .
      \label{eq:backus gilbert reconstruction constraints}
  \end{align}
\end{theorem}


From \eqref{eq:backus gilbert reconstruction constraints} we immediately
conclude that $\sM f = f$ for all $f \in \sG$. This formulation of the MLS
approximation is also called the Backus-Gilbert approximation and enables us to
state an explicit upper bound for the error of the MLS approximation. This upper
bound will depend on the error of the local best approximation of $f$.

Let us introduce some notation. The distance between two points
$y,z \in \S^{d-1}$ on the sphere is measured via the great circle distance
\[ d(y,z) = \arccos \left( \langle y,z \rangle \right) , \]
where $\langle y,z \rangle = y_1 z_1 + \dots + y_d z_d$ denotes the euclidean
inner product on $\R^d$. The open spherical cap $C_{\delta}(y)$ with radius
$\delta > 0$ around the center $y \in \S^{d-1}$ is defined as
\[ C_{\delta}(y) \coloneqq \left\{ z \in \S^{d-1} \, \middle\vert \, d(y,z) < \delta \right\} . \]

\begin{definition}
  Let $f \colon \S^{d-1} \to \C$ be a spherical function, $y \in \S^{d-1}$ and
  $\delta>0$.  Then the error of the local best approximation of $f$ from
  $\sG$ on the spherical cap $C_{\delta}(y)$ is defined as
  \begin{equation}
    \label{eq:error of local best
      approximation} E_{\delta}(\sG,f,y) \coloneqq \inf\limits_{g \in \sG} \sup\limits_{z \in C_{\delta}(y)} \vert f(z) - g(z) \vert
  \end{equation}
\end{definition}

A proof of the following theorem can be found in \cite{Fasshauer.2008,
  Levin.1998}.

\begin{theorem}\label{lem:MLS error estimate}
  Let $\sM f(y)$ be the MLS approximation of $f$ in $y \in \S^{d-1}$ and let
  \[ \delta \coloneqq \max\limits_{w(y_i,y) > 0} d(y_{i},y) \]
  denote the largest distance to the center $y$ among all nodes with positive
  weight. The error of the MLS approximation is bounded by
  \begin{align}\label{eq:MLS error estimate}
    \vert f(y) - \sM f(y) \vert \leq E_{\delta}(\sG,f,y)
    \left( 1 + \sum\limits_{w(y_i,y) > 0} \vert a_i^{\ast}(y) \vert \right) .
  \end{align}
\end{theorem}

The sum $\sum_{w(y_i,y) > 0} \vert a_i^{\ast}(y) \vert$ is called the
Lebesgue constant.

\edited{
  In order to analyze the asymptotic behavior of the MLS approximation we assume
  that the number of nodes $N$ goes to infinity while, simultaneously,
  a certain uniformity condition is satisfied, as specified in detail in
  \cref{sec:lebesgue-constant}.
  Furthermore, we adjust the weight function $w$ in dependency of the number of
  nodes $N$ such that for each $y$ the number of nodes $y_{i}$ with positive
  weight $w(y_{i},y_{i})$ remains approximately constant. Thus, asymptotically,
  the weight function $w$ becomes more and more localized and hence $\delta$
  converges to 0 in the above theorem.
}

It is easy to check, that $\delta \to 0$ also implies
$E_{\delta}(\sG,f,y) \to 0$ and the error $\vert \sM f(y) - f(y) \vert$ inherits this
decay if at the same time the Lebesgue constant remains bounded. In the
following section we will find for every $L \in \N_{0}$ an ansatz space $\sG$,
such that $E_{\delta}(\sG,f,y) = \sO \left( \delta^{L+1} \right)$ for
$\delta \to 0$. In \cref{sec:lebesgue-constant} we give sufficient conditions on
the weight function and on the nodes, which guarantee that the Lebesgue constant
is uniformly bounded for all centers $y \in \S^{d-1}$. Overall this will give the
desired approximation order
$\| \sM f - f \|_{\infty} = \sO \left( \delta^{L+1} \right)$ of the MLS
approximation.


\section{A Minimal Ansatz Space for Local Approximation on Spheres}
\label{sec:localApprox} 

The goal of this section is to find for given $L \in \N_0$ an ansatz space
$\sG \subset \sC^{\infty}(\S^{d-1})$, such that for all functions
$f \in \sC^{L+1}(\S^{d-1})$ and arbitrary center $y \in \S^{d-1}$ the error of the local
best approximation of $f$ on the spherical cap $C_{\delta}(y)$ with functions
from $\sG$ decays with order $\sO \left( \delta^{L+1} \right)$ for $\delta \to 0$, i.e.\
\begin{align}\label{eq:desired order of
approximation} E_\delta \left( \sG,f,y \right) = \sO \left( \delta^{L+1} \right) .
\end{align} Additionally the dimension of $\sG$ should be as small as possible.

Our idea is to utilize Taylor series to obtain the desired result. This requires
a local projection between the sphere $\S^{d-1}$ and its tangent space
$\R^{d-1}$. Thanks to symmetry it will be sufficient to consider the
following projection between $\R^{d-1}$ and the northern hemisphere, which is
just the spherical cap $C_{\frac {\pi}2}(e_{d})$ with radius $\frac {\pi} 2$,
centered around the north pole  $e_d = (0,\dots,0,1) \in \S^{d-1}$.

\begin{definition}\label{def:projection of points}
  The orthogonal projection from the open ball
  $B_1(\vc{0}) \subset \R^{d-1}$ with radius $1$ around $\vc{0} \in \R^{d-1}$ onto
  the open spherical cap $C_{\frac{\pi}{2}}(e_d) \subset \S^{d-1}$ with radius
  $\frac{\pi}{2}$ around the north pole $e_d \in \S^{d-1}$ is defined as
  \begin{equation}\label{eq:projection of points}
    \pi \colon B_{1}(\vc{0}) \to C_{\frac{\pi}{2}}(e_d) , \, (x_1,\dots,x_{d-1})
    \mapsto \left( x_1,\dots,x_{d-1},\sqrt{1-\|x\|_2^2} \right) .
  \end{equation}
\end{definition}

It is easy to check that this projection is bijective with inverse
\begin{equation}\label{eq:inverse projection of points}
  \pi^{-1} \colon C_{\frac{\pi}{2}}(e_d) \to B_1(\vc{0}) , \,
  (y_1,\dots,y_d) \mapsto (y_1,\dots,y_{d-1}) .
\end{equation}
For a given function $f \in \sC(\S^{d-1})$, the projection $\pi$ naturally
defines the pullback function
\[ f^{\ast} \coloneqq \concat{f}{\pi} \colon B_1(\vc{0}) \to \C , \] which
possesses the same order of differentiability as $f$ itself. Indeed, for
$f \in \sC^L(\S^{d-1})$ we obtain
\[ f^{\ast}(x) = f \circ \pi (x) = f \left( x_1,\dots,x_{d-1},\sqrt{1-\|x\|_2^2} \right) \]
and the variable substitution $x_d \mapsfrom \sqrt{1-\|x\|_2^2}$ does not affect
the order of differentiability for $x \in B_1(\vc{0})$. This allows us to
utilize the Taylor series of the pullback function $f^{\ast}$ as a substitute of
the Taylor series of $f$ itself. To this end we denote by
\begin{equation*}
  \Hom_{\ell}(\R^d) = \spn \left\{ \vc{x}^{\vc{\alpha}} \colon \R^d \to \R , \,
    x \mapsto \prod\limits_{i=1}^d x_i^{\alpha_i} \, \middle\vert \, \vc{\alpha} \in \N_0^d , \,
    \vert \vc{\alpha} \vert = \alpha_1 + \dots + \alpha_d = \ell \right\}
\end{equation*}
the space of homogeneous polynomials of degree $\ell \in \N_0$ on $\R^{d}$. The
dimension is given by
\begin{equation}
  \label{eq:dim of hom}
  \dim \left( \Hom_{\ell}(\R^{d}) \right) = \binom{d-1+\ell}{\ell} ,
\end{equation}
see \cite[Proposition 4.2]{Efthimiou.2014}. Further we denote their direct sum
up to degree $L \in \N_{0}$ by
\begin{equation*}
  \sP_L(\R^{d}) := \bigoplus_{\ell=0}^L \Hom_{\ell}(\R^{d})
\end{equation*}
\edited{and the dimension of $\sP_L(\R^{d})$ follows from the hockey-stick identity
\begin{equation}
  \label{eq:dim of sum of poly spaces}
  \dim \left( \sP_L(\R^{d}) \right)
  = \sum\limits_{\ell = 0}^{L} \binom{d-1+\ell}\ell
  = \binom{d-1+L+1}{L} = \binom{d+L}{L} .
\end{equation}
}

\begin{definition}\label{def:taylor of pullback}
For $L \in \N_0$ we define the mapping
\begin{align}\label{eq:definition of mapping
Phi} T_L \colon \sC^L(\S^{d-1}) \to \sP_L(\R^{d-1}) , \, f \mapsto T_L \left( f \right) = \sum\limits_{\vc{\alpha} \in \N_0^{d-1} , \, \vert \vc{\alpha} \vert \leq L} \frac{\left( f^{\ast} \right)^{(\vc{\alpha})}(\vc{0})}{\vc{\alpha}!} \vc{x}^{\vc{\alpha}} , \notag
\end{align} that maps every function $f \in \sC^L(\S^{d-1})$ to the Taylor
series of order $L$ of its pullback $f^{\ast} = \concat{f}{\pi}$.
\end{definition}

Note, that the truncated Taylor series $T_L(g)$ of arbitrary degree $L \in \N_0$
is well defined for all ansatz functions
$g \in \sG \subset \sC^{\infty}(\S^{d-1})$. Investigating the image
$T_L(\sG) \subset \sP_L(\R^{d-1})$ allows us to state a sufficient condition,
which ensures that the error $E_{\delta}(\sG,f,y)$ of the local best
approximation decays with order $\sO \left( \delta^{L+1} \right)$ for $\delta \to 0$.

\begin{lemma}
  \label{lem:Phi(G) = all polynomials up to degree L} Let $L \in \N_0$ be given
and let $\sG \subset \sC^{\infty}(\S^{d-1})$ be an ansatz space, such that
  \begin{enumerate}[label=\arabic*)]
  \item\label{item:rotation invariance} $\sG$ is rotational invariant, i.e.\
$\left\{ g \circ \mt{R} \, \middle\vert \, g \in \sG \right\} = \sG$ for all
rotation matrices $\mt{R} \in \SO{d}$ and
  \item\label{item:TL(G) spans all polynomials} every $p \in \sP_L(\R^{d-1})$ is
the Taylor series of the pullback function $g^{\ast}$ of some $g \in \sG$, i.e.\
          \begin{align} 
            T_L \left( \sG \right) = \sP_L(\R^{d-1}) .
    \end{align}
\end{enumerate} Then for all $f \in \sC^{L+1}(\S^{d-1})$ and $y \in \S^{d-1}$
the error $E_{\delta}(\sG,f,y)$ of the local best approximation of $f$ from the
ansatz space $\sG$ on the spherical cap $C_{\delta}(y)$ vanishes with order
$L+1$ as $\delta$ goes to $0$:
\[ E_{\delta}(\sG,f,y) = \sO(\delta^{L+1}) . \]
\end{lemma}

\begin{proof} First we prove the statement for the case $y = e_d$. To this end
  let $f \in \sC^{L+1}(\S^{d-1})$ be given. By assumption \ref{item:TL(G) spans
    all polynomials}, there exists a function $g \in \sG$, such that
  $T_L(g) = T_L(f)$. We show, that $g$ locally approximates $f$ with order
  $\sO \left( \delta^{L+1} \right)$ as $\delta$ goes to zero.

For $\delta \to 0$ we may assume $\delta < \frac{\pi}{2}$, such that the inverse
projection $\inv{\pi}(y)$ is well defined for all $y \in C_\delta(e_d)$. Let us
choose $y \in C_{\delta}(e_d)$ and set $x \coloneqq \pi^{-1}(y)$. From
$y \in C_{\delta}(e_d)$ we obtain $y_d > \cos\delta$ and the definition of the
projection $\pi$ yields:
\[ \| x\|_2^2 = y_1^2+\dots+y_{d-1}^2 = 1-y_d^2 < 1-\cos^2\delta = \sin^2\delta . \]
Thus $x$ is contained in the ball $B_{\sin\delta}(\vc{0})$ with radius
$\sin\delta$ around the center $\vc{0}$. Using the Taylor series of order
$L \in \N_0$ and the corresponding remainder term, we can write the pullback function
$f^{\ast}(x) = f \circ \pi (x)$ as
\[ f^{\ast}(x) = \left( T_L (f) \right)(x) + \sum\limits_{\vc{\alpha} \in \N_0^{d-1} , \, \vert \vc{\alpha} \vert = L+1} \frac{ \left( f^{\ast} \right)^{(\vc{\alpha})}(\xi_{f,x})}{\vc{\alpha}!} x^{\vc{\alpha}} \]
with some $\xi_{f,x} \in B_{\sin\delta}(\vc{0})$, that depends on the function
$f$ and the point $x$. The same can be applied to $g$. Because of $T_{L}(f) = T_{L}(g)$ this yields for
all $y \in C_\delta(e_d)$:
\begin{equation*}
  \left\vert f(y) - g(y) \right\vert = \left\vert f^{\ast}(x) - g^{\ast}(x) \right\vert =
  \left\vert \sum\limits_{\vc{\alpha} \in \N_0^{d-1} , \, \vert \vc{\alpha} \vert = L+1} \frac{(f^{\ast})^{(\vc{\alpha})}(\xi_{f,x}) - (g^{\ast})^{(\vc{\alpha})}(\xi_{g,x})}{\vc{\alpha}!} x^{\vc{\alpha}} \right\vert .
\end{equation*}
Since $x \in B_{\sin\delta}(\vc{0})$, each component of $x$ is bounded by
$\sin\delta$ and hence
$\vert \vc{x}^{\vc{\alpha}} \vert \leq \sin^{L+1} (\delta)$.
Further
\[ K := \sup\limits_{x \in B_{\sin\delta}(\vc{0})} \sum\limits_{\vc{\alpha} \in \N_0^{d-1} , \, \vert \vc{\alpha} \vert = L+1} \left\vert \frac{{\left(f^{\ast}\right)}^{(\vc{\alpha})}(\xi_{f,x}) - {\left(g^{\ast}\right)}^{(\vc{\alpha})}(\xi_{g,x})}{\vc{\alpha}!} \right\vert < \infty \]
because of $f, \, g \in \sC^{L+1}(\S^{d-1})$. This leads for $\delta \to 0$ to
the error estimate
\[ \sup\limits_{y \in C_{\delta}(e_d)} \vert f(y) - g(y) \vert \leq K \, \sin^{L+1} (\delta)  \leq K \, \delta^{L+1}  \]
and since $K$ is
decreasing as $\delta$ goes to zero this proves the statement
for the north pole $e_{d}$.
For general $y \in \S^{d-1}$ the statement follows from the rotational
invariance of $\sG$.
\end{proof}

A well known ansatz space $\sG$ that satisfies the conditions of
\Cref{lem:Phi(G) = all polynomials up to degree L} are the spherical
polynomials
\begin{equation*}
 \sH_L(\S^{d-1}) := \bigoplus_{\ell=0}^L \Harm_{\ell}(\S^{d-1}) 
\end{equation*}
up to a certain degree $L \in \N$ where
\begin{align*}
  \Harm_{\ell}(\S^{d-1})
  = \left\{ h\big\vert_{\S^{d-1}} \colon \S^{d-1} \to \C \, \middle\vert \, h \in \Hom_{\ell}(\R^d) , \, \Delta_d(h) \equiv 0 \right\}
\end{align*}
denotes the harmonic space of degree $\ell \in \N$ and $\Delta_d$
the Laplace operator in $d$ dimensions. The elements of $\Harm_{\ell}(\S^{d-1})$
are called spherical harmonics of degree $\ell$.

Let us collect some properties of the harmonic spaces. The dimension of one
harmonic space is given by
\begin{align}\label{eq:dim of harm}
  \dim \left( \Harm_{\ell}(\S^{d-1}) \right) =
  \dim \left( \Hom_{\ell}(\R^{d-1}) \right) +
  \dim \left( \Hom_{\ell-1}(\R^{d-1}) \right) , \, \ell \geq 1
\end{align}
and $\dim \left( \Hom_{0}(\R^{d-1}) \right) = 1$ respectively, see \cite[Equation
7]{Muller.1966}. By Theorem 5.7 from \cite{Axler.2001}, every
$p \in \Hom_{L}(\R^d)$ can be uniquely written as
\begin{equation*} p = h_{L} + \|x\|_2^2 \, h_{L-2} + \dots + \|x\|_2^{L-\sigma} h_{\sigma} ,
\end{equation*} where $\sigma \coloneqq L \pmod*{2}$ and
$h_{\ell} \in \Hom_{\ell}(\R^d)$ with
$\Delta_{d} \, h_{\ell} = 0 , \, \ell = L , L-2 , \dots , \sigma$. Restricting
the above representation of $p$ to the sphere proves the following lemma.

\begin{lemma}\label{lem:harmonic space contains all spherical polynomials} For
every monomial $\vc{x}^{\vc{\alpha}} \in \Hom_L(\R^d)$ of degree
$L = \vert \vc{\alpha} \vert = \alpha_1 + \dots + \alpha_d$
\[ \vc{x}^{\vc{\alpha}}(x) = \prod\limits_{i=1}^d x_i^{\alpha_i} , \quad x \in \R^d \]
there exist spherical harmonics $h_{\ell} \in \Harm_{\ell}(\S^{d-1})$ of degree
$\ell = L,L-2,\dots, \sigma \coloneqq L \pmod*{2}$, such that the restriction of
$\vc{x}^{\vc{\alpha}}$ to the sphere equals the sum of the spherical harmonics
$h_{\ell}$:
\[ \vc{x}^{\vc{\alpha}}(y) = h_L(y) + h_{L-2}(y) + \dots + h_{\sigma}(y) , \quad y \in \S^{d-1} . \]
\end{lemma}

Using this lemma we can easily verify that
$T_L \left( \sH_{L}(\S^{d-1}) \right) = \sP_L(\R^{d-1})$ and of course
$\sH_L(\S^{d-1})$ inherits the rotational invariance from the harmonic spaces,
hence $E_{\delta}(\sG,f,y) = \sO \left( \delta^{L+1} \right)$ for all
$f \in \sC(\S^{d-1})$ and for all $y \in \S^{d-1}$ as $\delta \to 0$ because of
\Cref{lem:Phi(G) = all polynomials up to degree L}.

However, we also want the dimension of the ansatz space to be as small as
possible. With respect to the condition
$T_L \left( \sG \right) = \sP_L(\R^{d-1})$ of \Cref{lem:Phi(G) = all polynomials
  up to degree L}, the dimension of the ansatz space is bounded from below by
\edited{
\begin{equation}
  \label{eq:lower bound on dim of G}
  \dim \left( \sG \right) \geq \dim \left( \sP_L(\R^{d-1}) \right) = \binom{d-1+L}L .
\end{equation}
}
In comparison, from \eqref{eq:dim of harm} and \eqref{eq:dim of sum of poly spaces}
we obtain
\edited{
\begin{align*}
  \dim \left( \sH_L(\S^{d-1}) \right) &=
  \dim \left( \sP_L(\R^{d-1}) \right) + \dim \left( \sP_{L-1}(\R^{d-1}) \right) \\
  &= \binom{d-1+L}L + \binom{d-1+L-1}{L-1}
\end{align*}
}
for $L \geq 1$ and $\dim \left( \sH_{0}(\S^{d-1}) \right) = 1$ respectively, which is almost
twice as big as the lower bound.

We will now drastically reduce the dimension of the ansatz space by leaving out the harmonic
spaces of every second degree and showing, that both conditions of \Cref{lem:Phi(G) = all
polynomials up to degree L} remain satisfied.

\begin{definition} For $L \in \N_0$ we denote the direct sum of the harmonic
spaces of even, respectively odd, degree up to $L$ by
\begin{align}\label{eq:even/odd spherical
harmonics} \sH_{L,2}(\S^{d-1}) \coloneqq \bigoplus\limits_{\ell \leq L , \, \ell \equiv L \pmod*{2}} \Harm_{\ell}(\S^{d-1}) .
\end{align}
\end{definition}

We can use \eqref{eq:dim of harm} to easily verify that the dimension of this new ansatz
space even \edited{attains the lower bound given in \eqref{eq:lower bound on dim of G}}
\begin{align}\label{eq:minimal dimension}
  \dim \left( \sH_{L,2}(\S^{d-1}) \right) = \dim \left( \sP_L(\R^{d-1}) \right)
  = \binom{d-1+L}{L} .
\end{align}
Since $\sH_{L,2}(\S^{d-1})$ is the direct sum of harmonic spaces, it
is still rotational invariant and thus we only have to show, that the condition
$T_L \left( \sH_{L,2}(\S^{d-1}) \right) = \sP_L(\R^{d-1})$ remains satisfied. To
this end we introduce a basis of both $\sH_{L,2}(\S^{d-1})$ and $\sP_L(\R^{d-1})$. For the
latter one, we utilize the monomial basis
\begin{align}\label{eq:monomial
basis} \left\{ \vc{x}^{\vc{\alpha}} \colon \R^{d-1} \to \R \, \middle\vert \, \vc{\alpha} \in \N_0^{d-1} , \, \vert \vc{\alpha} \vert \leq L \right\} .
\end{align} The harmonic spaces are usually equipped with an $L^2$-orthogonal
basis. However, for our purpose the following monomial basis of
$\sH_{L,2}(\S^{d-1})$ is more suitable.

\begin{lemma}\label{lem:basis of G_L} Let $L \in \N_0$ be given.
  The set of $d$-variate monomials
  \[ \sB_{L,2} \coloneqq \left\{ \vc{y}^{\vc{\alpha}} \colon \S^{d-1} \to \R \,
      \middle\vert \, \vc{\alpha} \in \N_0^d , \, \vert \vc{\alpha} \vert \leq L , \,
      \vert \vc{\alpha} \vert \equiv L \pmod*{2} , \, \alpha_d \leq 1 \right\} \]
  constitutes a basis of $\sH_{L,2}(\S^{d-1})$.
\end{lemma}

\begin{proof} First we are going to verify, that the number of elements in
$\sB_{L,2}$ equals the dimension of the space $\sH_{L,2}(\S^{d-1})$. After that
we show, that $\sH_{L,2}(\S^{d-1})$ is the span of $\sB_{L,2}$. Together this
implies, that $\sB_{L,2}$ is  a basis of $\sH_{L,2}(\S^{d-1})$. For the first
step, we count the elements in $\sB_{L,2}$, which gives
\begin{align}
  \label{eq:size of basis}
  \begin{split}
    \vert \sB_{L,2} \vert
    &= \left\vert \left\{ \vc{\alpha} \in \N_0^d \, \middle\vert \,
      \vert \vc{\alpha} \vert \leq L , \, \vert \vc{\alpha} \vert \equiv L \pmod*{2} , \,
      \alpha_d = 0 \right\} \right\vert \\
    &\quad \hspace{3cm} + \left\vert \left\{ \vc{\alpha} \in \N_0^d \, \middle\vert \,
      \vert \vc{\alpha} \vert \leq L , \, \vert \vc{\alpha} \vert \equiv L \pmod*{2} , \,
      \alpha_d = 1 \right\} \right\vert \\
    &= \left\vert \left\{ \widetilde{\vc{\alpha}} \in \N_0^{d-1} \, \middle\vert \,
      \vert \widetilde{\vc{\alpha}} \vert \leq L ,
      \vert \widetilde{\vc{\alpha}} \vert \equiv L \pmod*{2} \right\} \right\vert \\
    &\quad \hspace{3cm} + \left\vert \left\{ \widetilde{\vc{\alpha}} \in \N_0^{d-1} \,
      \middle\vert \, \vert \widetilde{\vc{\alpha}} \vert \leq L-1 , \,
      \vert \widetilde{\vc{\alpha}} \vert \equiv L-1 \pmod*{2} \right\} \right\vert \\
    &= \left\vert \left\{ \widetilde{\vc{\alpha}} \in \N_0^{d-1} \, \middle\vert \,
      \vert \widetilde{\vc{\alpha}} \vert \leq L \right\} \right\vert =
      \dim \left( \sP_L(\R^{d-1}) \right) = \dim \left( \sH_{L,2}(\S^{d-1}) \right) ,
  \end{split}
\end{align}
where we used \eqref{eq:minimal dimension} in the last step.

For the second part of the proof we first note that
$\spn \left\{ \sB_{L,2} \right\} \subseteq \sH_{L,2}(\S^{d-1})$ by
\Cref{lem:harmonic space contains all spherical polynomials}. For the other
inclusion, let $h \in \sH_{L,2}(\S^{d-1})$ be given and denote
$\sigma \coloneqq L \mod 2$. By the definition of $\sH_{L,2}(\S^{d-1})$ we may
write $h = h_{L} + h_{L-2} + \dots + h_{\sigma}$ with
$h_{\ell} \in \Harm_{\ell}(\S^{d-1}) , \, \ell = L,L-2,\dots,\sigma$. Further,
by the definition of the harmonic spaces, there exist homogeneous polynomials
$H_{\ell} \in \Hom_{\ell}(\R^{d})$, such that
$h_{\ell} = H_{\ell} \big\vert_{\S^{d-1}}$ for $\ell = L,L-2,\dots,\sigma$ and we
may write the sum as
\begin{equation*}
  H \coloneqq H_{L} + H_{L-2} + \dots +  H_{\sigma} = \sum\limits_{\vc{\alpha} \in \N_{0}^{d} , \, \vert \vc{\alpha} \vert \leq L , \, \vert \vc{\alpha} \vert \equiv L \pmod*{2}} c_{\alpha} \vc{y}^{\vc{\alpha}}
\end{equation*}
with coefficients $c_{\alpha} \in \C$. If the coefficients are $0$ for all
monomials $\vc{y}^{\vc{\alpha}}$ with $\alpha_{d} \geq 2$ then there exists an
element $g \in \spn \left\{ \sB_{L,2} \right\}$ with
$g = H \big\vert_{\S^{d-1}} =  h$ and we are done. If
not, we repeat substitution of the identity $y_{d}^{2} = 1 - y_1^2 + \dots + y_d^2$
\begin{align*} \vc{y}^{\vc{\alpha}}(y) &= y_1^{\alpha_1} \cdot \dots \cdot y_d^{\alpha_d} = y_1^{\alpha_1} \cdot \dots \cdot y_{d-1}^{\alpha_{d-1}} y_d^{\alpha_d-2} \left( 1 - y_1^2 - \dots - y_{d-1}^2 \right) \\ &= y_1^{\alpha_1} \cdot \dots \cdot y_{d-1}^{\alpha_{d-1}} y_d^{\alpha_d-2} - \sum\limits_{i=1}^{d-1} y_i^2 \cdot y_1^{\alpha_{1}} \cdot \dots \cdot y_{d-1}^{\alpha_{d-1}} y_d^{\alpha_d-2} , \quad y \in \S^{d-1}
\end{align*}
until the highest remaining exponent in the $d$-th variable is $1$ or $0$. Note
that the value of $H$ remains unchanged on $\S^{d-1}$. Since the substitution
always yields monomials of the same degree or $2$ less, the condition
$\vert \vc{\alpha} \vert \leq L , \, \vert \vc{\alpha} \vert \equiv L \pmod*{2}$
remains satisfied in all monomials $\vc{y^{\vc{\alpha}}}$ too. Hence there
exists an element $g \in \spn \left\{ \sB_{L,2} \right\}$ with
$g = H \big\vert_{\S^{d-1}} = h$
and thus $\sH_{L,2} \subseteq \spn \left\{ \sB_{L,2} \right\}$.
\end{proof}

\begin{remark} \label{rem:monmoial basis of spherical harmonics} It is
straightforward to see, that
$\sH_L(\S^{d-1}) = \sH_{L,2}(\S^{d-1}) \oplus \sH_{L-1,2}(\S^{d-1})$ and hence
the union $\sB_{L,2} \cup \sB_{L-1,2}$ constitutes a monomial basis of
$\sH_L(\S^{d-1})$. This basis has also been mentioned in \cite{Golitschek.2001}.
\end{remark}

The monomial basis $\sB_{L,2}$ enables us to show, that the second
condition of \Cref{lem:Phi(G) = all polynomials up
to degree L} is satisfied as well.

\begin{lemma}\label{lem:T_L circ P is bijective} Let $L \in \N_0$ be given. The
linear mapping
\[ T_L \colon \sH_{L,2}(\S^{d-1}) \to \sP_L(\R^{d-1}) , \quad g \mapsto \sum\limits_{\vc{\alpha} \in \N_0^{d-1} , \, \vert \vc{\alpha} \vert \leq L} \frac{\left( g^{\ast} \right)^{(\vc{\alpha})} (\vc{0})}{\vc{\alpha}!} \vc{x}^{\vc{\alpha}} , \]
that maps every function $g \in \sH_{L,2}(\S^{d-1})$ to the truncated Taylor
series of order $L$ of the pullback function $g^{\ast} = g \circ \pi$ satisfies
condition \ref{item:TL(G) spans all polynomials} of \Cref{lem:Phi(G) = all
polynomials up to degree L}, i.e.\
$T_L \left( \sH_{L,2}(\S^{d-1}) \right) = \sP_L(\R^{d-1})$.
\end{lemma}

\begin{proof} We equip $\sH_{L,2}(\S^{d-1})$ with the basis $\sB_{L,2}$ from
\Cref{lem:basis of G_L} and $\sP_L(\R^{d-1})$ with the monomial basis
from \eqref{eq:monomial basis}. Let
$\mt{M}_L = m_L^{\vc{\alpha},\vc{\beta}}$ denote the
corresponding matrix representation of $T_L$, where
$\vc{\alpha} \in \N_0^d , \, \vert \vc{\alpha} \vert \leq L , \, \vert \vc{\alpha} \vert \equiv L \pmod*{2} , \, \alpha_d \leq 1$
and $\vc{\beta} \in \N_0^{d-1} , \, \vert \vc{\beta} \vert \leq L$. From
\begin{align*} T_L (\vc{y}^{\vc{\alpha}}) = \sum\limits_{\vc{\beta} \in \N_0^{d-1} , \, \vert \vc{\beta} \vert \leq L} \frac{\left( \vc{y}^{\vc{\alpha}} \circ \pi \right)^{(\vc{\beta})}(\vc{0})}{\vc{\beta}!} \vc{x}^{\vc{\beta}}
\end{align*} we can read off the matrix entries of $\mt{M}_{L}$:
\begin{align}\label{eq:matrix entries of
M} m_L^{\vc{\alpha},\vc{\beta}} = \frac{\left( \vc{y}^{\vc{\alpha}} \circ \pi \right)^{(\vc{\beta})}(\vc{0})}{\vc{\beta}!} .
\end{align}
Let $\vc{y}^{\vc{\alpha}} \in \sB_{L,2}$ be a basis element and denote
$\widetilde{\vc{\alpha}} := (\alpha_1,\dots,\alpha_{d-1}) \in \N_0^{d-1}$ and
$\vc{\alpha} = (\widetilde{\vc{\alpha}},\alpha_d)$. Recall that, by the
definition of $\sB_{L,2}$, the multi-index $\vc{\alpha}$  satisfies
$\alpha_d \in \{0,1\}$. We are going to compute
$m_L^{\vc \alpha,\vc \beta}$ for $\alpha_d = 0$ and $\alpha_d = 1$ separately.

First we consider $\alpha_d = 0$. For $x \in B_1(\vc{0})$ the pullback of the
basis element $\vc{y}^{(\widetilde{\vc{\alpha}},0)} \in \sB_{L,2}$ can be
written as
\[ \vc{y}^{(\widetilde{\vc{\alpha}},0)} \circ \pi (x) = \vc{y}^{(\widetilde{\vc{\alpha}},0)} (x_1,\dots,x_{d-1},\sqrt{1-\|x\|_2^2}) = \left( x_{1}, \dots , x_{d-1} \right)^{\widetilde{\vc{\alpha}}} = x^{\widetilde{\vc{\alpha}}} , \]
hence
$\vc{y}^{(\widetilde{\vc{\alpha}},0)} \circ \pi = \vc{x}^{\vc{\widetilde{\alpha}}}$
on $B_{1}( \vc{0} )$.
In order to evaluate the matrix entry \eqref{eq:matrix entries of M}, we have to
take the derivative with respect to some
$\vc{\beta} \in \N_0^{d-1} , \, \vert \vc{\beta} \vert \leq L$ and evaluate it
in the point $\vc{0} \in \R^{d-1}$. Since the evaluation of any non-constant
monomial in the point $\vc{0} \in \R^{d-1}$ always gives $0$, we get
\begin{align}\label{eq:matrix entries for alpha_d = 0}
  m_L^{(\widetilde{\vc{\alpha}},0),\vc{\beta}} = \frac{\left( \vc{y}^{(\widetilde{\vc{\alpha}},0)} \circ \pi \right)^{(\vc{\beta})}(\vc{0})}{\vc{\beta}!} = \frac{\left( \vc{x}^{\widetilde{\vc{\alpha}}} \right)^{(\vc{\beta})} \left( \vc{0} \right)}{\vc{\beta}!} = \begin{cases} 1 &, \widetilde{\vc{\alpha}} = \vc{\beta} \\ 0 &, \widetilde{\vc{\alpha}} \neq \vc{\beta} \end{cases} .
\end{align}

Now we consider the case $\alpha_d = 1$. For $x \in B_1(\vc{0})$ the pullback of
the basis element $\vc{y}^{(\widetilde{\vc{\alpha}},1)} \in \sB_{L,2}$ can be
written as
\[ \vc{y}^{(\widetilde{\vc{\alpha}},1)} \circ \pi (x) = \vc{y}^{(\widetilde{\vc{\alpha}},1)} \left( x_1,\dots,x_{d-1},\sqrt{1-\|x\|_2^2} \right) = x^{\widetilde{\vc{\alpha}}} \, \sqrt{1-\|x\|_2^2} . \]
Let us denote $W(x) := \sqrt{1-\|x\|_2^2}$. The Leibniz rule allows us to
compute the derivative with respect to $\vc{\beta} \in \N_0^{d-1}$ as follows:
\begin{align}\label{eq:diff with Leibniz rule}
  \begin{split}
    \left( \vc{x}^{\widetilde{\vc{\alpha}}} \, W \right) ^{(\vc{\beta})} \left( \vc{0} \right)
    &=
      \left( \sum\limits_{\vc{\gamma} \in \N_0^{d-1} , \, \vc{\gamma} \leq \vc{\beta}}
      \binom{\vc{\beta}}{\vc{\gamma}}
      \left(  \vc{x}^{\widetilde{\vc{\alpha}}} \right)^{(\vc{\gamma})}
      W^{(\vc{\beta} - \vc{\gamma})} \right) \left( \vc{0} \right)
    \\ &=
         \begin{cases}
           \binom{\vc{\beta}}{\widetilde{\vc{\alpha}}} \, \widetilde{\vc{\alpha}}! \,
           W^{(\vc{\beta}-\widetilde{\vc{\alpha}})} \left( \vc{0} \right)
           &, \widetilde{\vc{\alpha}} \leq \vc{\beta}
           \\
           \hfil 0 &, \widetilde{\vc{\alpha}} \nleq \vc{\beta}
         \end{cases} ,
  \end{split}
\end{align}
where we used \eqref{eq:matrix entries for alpha_d = 0} in the second step.
We want to show that the matrix $\mt{M}_{L}$ is similar to a lower triangle
matrix with ones on the diagonal. Then
$\vert \det \left( \mt{M}_{L} \right) \vert = 1$ and thus $T_{L}$ is bijective.
To this end it will be sufficient to evaluate \eqref{eq:diff with Leibniz rule}
for $\vert \widetilde{\vc{\alpha}} \vert \geq \vert \vc{\beta} \vert$. But then
$\widetilde{\vc{\alpha}} \leq \vc{\beta}$ never happens, except if
$\widetilde{\vc{\alpha}} = \vc{\beta}$, thus we obtain
\begin{align}
  \label{eq:matrix entries for alpha_d = 1}
  m_L^{(\widetilde{\vc{\alpha}},1),\vc{\beta}} = \frac{\left( \vc{x}^{\widetilde{\vc{\alpha}}} \, W \right) ^{(\vc{\beta})}(\vc{0})}{\vc{\beta}!} = \begin{cases} 1 &, \widetilde{\vc{\alpha}} = \vc{\beta} \\ 0 &, \widetilde{\vc{\alpha}} \neq \vc{\beta} , \, \vert \widetilde{\vc{\alpha}} \vert \geq \vert \vc{\beta} \vert \end{cases} .
\end{align}

Let us identify the rows and columns with the multi-indice
$\vc{\beta} \in \N_0^{d-1}$ corresponding to the derivatives and the exponents
$\vc{\alpha} \in \N_0^d$ of the basis elements $\vc{y}^{\vc{\alpha}} \in \sB_L$,
respectively. Using the formulas \eqref{eq:matrix entries for alpha_d = 0} and
\eqref{eq:matrix entries for alpha_d = 1} we rearrange the rows and columns of
$\mt{M}_{L}$ in order to obtain a lower triangle matrix with ones on the diagonal.
Let $\vc{\alpha}_1,\dots,\vc{\alpha}_{N_L}$ be an arrangement of
the columns, such that $\vert \vc{\alpha}_j \vert$ increases or stays constant,
as the column index $j$ increases. Additionally we arrange the rows such that
$\vc{\beta}_j = \widetilde{\vc{\alpha}}_j$ for $j = 1,\dots,N_L$. It can be seen
from \eqref{eq:size of basis} that this is possible. Now $j \geq i$ implies
$\vert \widetilde{\vc{\alpha}}_j \vert \geq \vert \vc{\beta}_i \vert$ and by
applying the formulas \eqref{eq:matrix entries for alpha_d = 0} and
\eqref{eq:matrix entries for alpha_d = 1} we thus get
\[ m_L^{\vc{\alpha}_j,\vc{\beta}_i} = \begin{cases} 1 &, j = i \\ 0 &, j > i \end{cases} , \]
hence the described ordering leads to a lower triangle matrix with ones
on the main diagonal.
\end{proof}

Combining the previous results we obtain the concluding theorem of this section.

\begin{theorem}
  \label{thm:main result} Let $L \in \N_0$ be given and let
  \begin{equation*} \sH_{L,2}(\S^{d-1}) = \bigoplus\limits_{\ell \leq L , \, \ell \equiv L \pmod*{2}} \Harm_{\ell}(\S^{d-1})
  \end{equation*} be the direct sum of the harmonic spaces of even, respectively
odd, degree up to $L$. For all $f \in \sC^{L+1}(\S^{d-1})$ and $y \in \S^{d-1}$
the error $E_{\delta}(\sH_{L,2}(\S^{d-1}),f,y)$ of the local best approximation
of $f$ on the spherical cap $C_{\delta}(y)$ from $\sH_{L,2}(\S^{d-1})$ decays
with order $\sO(\delta^{L+1})$ as $\delta$ goes to zero:
  \begin{equation*} E_{\delta}(\sH_{L,2}(\S^{d-1}),f,y) = \sO(\delta^{L+1}) .
  \end{equation*}
\end{theorem}

\begin{proof} Since $\sH_{L,2}(\S^{d-1})$ is the direct sum of harmonic spaces,
it is rotational invariant. Further, \Cref{lem:T_L circ P is bijective} states
$T_L \left( \sH_{L,2}(\S^{d-1}) \right) = \sP_L(\R^{d-1})$. The assertion
follows directly from \Cref{lem:Phi(G) = all polynomials up to degree L}.
\end{proof}

 
\section{A Uniform Bound for the Lebesgue Constant}
\label{sec:lebesgue-constant}

Let a function $f \in \sC^{L+1}(\S^{d-1})$ and a center $y \in \S^{d-1}$ be given and let $\sM f(y)$ be the MLS approximation of $f$ in $y$ with respect to the ansatz space $\sG \subset \sC(\S^{d-1})$. 
If the weight function $w$ vanishes for $d(y_{i},y) \geq \delta$, then by
\Cref{lem:MLS error estimate} the error of the MLS approximation can be estimated by
\begin{align}\label{eq:error est of MLS} 
\vert f(y) - \sM f(y) \vert \leq E_{\delta}(\sG,f,y) \left( 1 + \sum\limits_{w(y_i,y)>0} \vert a_i^{\ast}(y) \vert \right) , 
\end{align} 
where the coefficients $a_i^{\ast}(y) \in \C$ are the unique solution of the minimization problem
\begin{align} 
&\min\limits_{a_i(y) \in \C} \sum\limits_{w(y_i,y)>0} \frac{1}{w(y_i,y)} \vert a_i(y) \vert^2 \label{eq:target function} \\ 
\text{subject to} \quad &\sum\limits_{w(y_i,y)>0} a_i(y) g(y_i) = g(y) , \quad g \in \sG . \label{eq:reconstruction constraints}
\end{align}

\begin{definition}
  The Lebesgue constant with respect to the ansatz space $\sG$, the weight
  function $w$, the nodes $Y = \left\{ y_1,\dots,y_N \right\} \subset \S^{d-1}$
  and the center $y \in \S^{d-1}$ is defined as
  \begin{align}\label{eq:def Lebesgue constant}
    \sL \left( \sG,w,Y,y \right) \coloneqq
    \sum\limits_{w(y_i,y) > 0} \vert a_i^{\ast}(y) \vert ,
  \end{align}
  where the dependencies are defined by the optimization problem
  \eqref{eq:target function}, \eqref{eq:reconstruction constraints}.
\end{definition}

Recall, that we are interested in the decay rate of the error 
\[ \| f - \sM f \|_{\infty} = \sup\limits_{y \in \S^{d-1}} \vert f(y) - \sM f(y) \vert \] 
as the number $N$ of nodes increases to $\infty$.
In order to keep the size of the local minimization problem \eqref{eq:target
  function}, \eqref{eq:reconstruction constraints} bounded, we simultaneously
have to make the weight function $w$ more localized, thus we may assume
$\delta \to 0$.
From \Cref{thm:main result} we know, that the error
$E_{\delta}\left( \sH_{L,2}(\S^{d-1}),f,y \right)$ of the local best
approximation vanishes with order $\sO \left( \delta^{L+1} \right)$ for
$\delta \to 0$. By the error estimate \eqref{eq:error est of MLS}, the MLS
approximation inherits this approximation order, if the corresponding Lebesgue
constant is uniformly bounded for all $y \in \S^{d-1}$. Thus, the goal of this
section is to state for all $y \in \S^{d-1}$ a uniform bound $C>0$ of the
Lebesgue constant with respect to the ansatz space  $\sH_{L,2}(\S^{d-1})$, i.e.\
\[ \sL \left( \sH_{L,2}(\S^{d-1}),w,Y,y \right) \leq C . \]
To this end we have to impose some assumptions on the weight function $w$ and on
the distribution of the nodes $Y = \left\{ y_1,\dots,y_N \right\}$.
As proposed in \cite{Wendland.2001} we consider radial weight functions of the
form
\begin{align}\label{eq:weight function} 
  w \colon \S^{d-1} \times \S^{d-1} \to \lbrack 0,\infty ) , \quad
  w(y, z) = \phi \left( \frac{d(y,z)}{\delta} \right) ,
\end{align} 
where $\phi \colon \lbrack 0,\pi \rbrack \to \lbrack 0,\infty)$ is a continuous function with
\begin{align}
\label{eq:weight fun prop 1}\phi(r) &> 0 , \quad 0 \leq r \leq \frac{1}{2} , \\
\label{eq:weight fun prop 2}\phi(r) &= 0 , \quad r \geq 1 .
\end{align}

The radiality is not necessary, but it simplifies the following considerations.
Note, that property \eqref{eq:weight fun prop 2} ensures $w(y_{i},y) = 0$ whenever
$d(y_{i},y) \geq \delta$, which we will later combine with
\Cref{lem:MLS error estimate}. On the other hand, property \eqref{eq:weight fun
prop 1} guarantees that the support of the weight function is not too small.
Its necessity will become clear in the proof of \Cref{lem:target function
bounded}. Let us now discuss the nodes.

\begin{definition}\label{def:grid properties} 
Let $y_1,\dots,y_N \in \S^{d-1}$ be nodes on the sphere. 
The fill distance 
\[ h := \sup\limits_{y \in \S^{d-1}} \min\limits_{i=1,\dots,N} d(y_i,y) \] 
is the maximal distance of a point $y \in \S^{d-1}$ to the nearest node $y_i$ and the separation distance 
\[ q \coloneqq \frac12 \min\limits_{i \neq j} d(y_i,y_j) \] 
is half of the minimal distance between $2$ nodes. 
\end{definition}

With the restriction to weight functions of the form \eqref{eq:weight function}
and the reduction of the nodes $Y = \left\{ y_1,\dots,y_N \right\}$ to their fill
distance $h$ and separation distance $q$, we are now prepared to deal with the
Lebesgue constant. Our starting point is the estimate
\begin{align}\label{eq:CSU} 
\sL \left( \sG,w,Y,y \right) \leq \left( \sum\limits_{w(y_i,y)>0} \frac{1}{w(y_i,y)} \vert a_i^{\ast}(y) \vert^2 \right)^{\frac12} \, \left( \sum\limits_{w(y_i,y)>0} w(y_i,y) \right)^{\frac12} .
\end{align} 
We are going to bound both factors individually and we start with the first one, which is exactly the minimal value of the target function \eqref{eq:target function} with respect to the reconstruction constraints \eqref{eq:reconstruction constraints}.

The statement of the following lemma can be found in \cite{Wendland.2001} for the special case $\sG = \sH_L(\S^{d-1})$. 
Here we present a generalization for arbitrary subspaces $\sG \subseteq \sH_L(\S^{d-1})$.
For the proof of this generalization we can almost use the exact same argumentation as in \cite{Wendland.2001}. 
We still think that it is important to present the proof here, such that one can
see where the generalization takes place and also  understand the necessity of the assumptions.

\begin{lemma}\label{lem:target function bounded} 
Let $\phi \colon [0,\pi] \to \lbrack0,\infty)$ be a continuous function
satisfying \eqref{eq:weight fun prop 1}, \eqref{eq:weight fun prop 2} and let
\begin{equation*}
  w(y,z) = \phi \left( \frac{d(y,z)}{\delta} \right)
\end{equation*}
be the corresponding weight function. For every subspace
$\sG \subseteq \sH_L(\S^{d-1})$ there exist constants $h_0,r,C(\phi) > 0$, such
that for arbitrary nodes $y_1,\dots,y_N$ with fill distance $h \leq h_0$ and any
$\delta \geq r \, h$ the optimal value of the minimization problem
\eqref{eq:target function}, \eqref{eq:reconstruction constraints} satisfies
\begin{align}\label{eq:target function bounded}
\sum\limits_{w(y_i,y)>0} \frac{1}{w(y_i,y)} \vert a_i^{\ast}(y) \vert^2 \leq C(\phi) \quad , \, y \in \S^{d-1} .
\end{align} 
\end{lemma} 

\begin{proof} 
By \cite[Theorem 1.4]{KurtJetter.1999} there exist for every $L \in \N_0$ constants $h_0,C_1,C_2 > 0$, such that for arbitrary nodes $y_1,\dots,y_N \in \S^{d-1}$ with fill distance $h \leq h_0$ there exist coefficients $\hat{a}_i(y) , \, i=1,\dots,N$ with the following properties: 
\begin{enumerate}[topsep=0pt, label=\arabic*)]
  \item\label{item:coeffs prop 1} The reproduction constraints \eqref{eq:reconstruction constraints} are satisfied, i.e.\
        \[ \sum\limits_{i=1}^N \hat{a}_i(y) g(y_i) = g(y) \quad \text{for all} \quad g \in \sH_L(\S^{d-1}) . \]
  \item\label{item:coeffs prop 2} The coefficient functions
        $\hat{a}_{i} \colon \S^{d-1} \to \R, \ i = 1 , \dots , N$ are compactly supported, i.e.\
        \begin{equation*}
          \hat{a}_i(y) = 0 \text{ for all } y \in \S^{d-1} \text{ with } d(y_i,y) > C_1 \, h.
        \end{equation*}
  \item\label{item:coeffs prop 3} The sum of the absolute values is uniformly bounded by $C_2$: $\sum\limits_{i=1}^N \vert \hat{a}_i(y) \vert < C_2$.
\end{enumerate} 
This is where the generalization to arbitrary subspaces of $\sH_L(\S^{d-1})$ takes place. 
The property \ref{item:coeffs prop 1} is satisfied for the whole space $\sH_L(\S^{d-1})$, thus it is certainly satisfied for the subspace $\sG$. 
Since the coefficients $a_i^{\ast}(y)$ of the MLS approximation minimize the target function \eqref{eq:target function} with respect to the reconstruction constraints \eqref{eq:reconstruction constraints}, we obtain: 
\begin{align*} 
  \sum\limits_{w(y_i,y)>0} \frac{\vert a_i^{\ast}(y) \vert^2}{w(y_i,y)} \leq \sum\limits_{w(y_i,y)>0} \frac{\vert \hat{a}_i(y) \vert^2}{w(y_i,y)} .
\end{align*} 
We now choose the constant $r$ as $r = 2 \, C_1$ . 
For arbitrary $\delta \geq r \, h$ this yields
\[ C_1 \, h = \frac{2 \, C_1 \, h}2 = \frac{r \, h}2 \leq \frac{\delta}2 , \] 
such that property \ref{item:coeffs prop 2} of the coefficients $\hat{a}_i(y)$ implies $\hat{a}_i(y) = 0$ for all nodes $y_i$ with $d(y_i,y) > \frac{\delta}2$.
By property \eqref{eq:weight fun prop 1} of the function $\phi$ that defines the weight function $w$, we obtain
\[ \min\limits_{d(y_i,y)\leq\frac{\delta}{2}} w(y_i,y) = \min\limits_{d(y_i,y)\leq\frac{\delta}{2}} \phi \left( \frac{d(y_i,y)}{\delta} \right) = \min\limits_{t \in [0,\frac12]} \phi(t) =: c_{\phi} > 0 . \] 
This allows us to continue the estimation of the target function via 
\[ \sum\limits_{w(y_i,y)>0} \frac{\vert \hat{a}_i(y) \vert^2}{w(y_i,y)} = \sum\limits_{d(y_i,y)\leq\frac{\delta}{2}} \frac{\vert \hat{a}_i(y) \vert^2}{w(y_i,y)} \leq \frac{1}{c_{\phi}} \sum\limits_{d(y_i,y) \leq \frac{\delta}{2}} \vert \hat{a}_i(y) \vert^2 \] 
and putting everything together leads to 
\[ \sum\limits_{w(y_i,y)>0} \frac{\vert a_i^{\ast}(y) \vert^2}{w(y_i,y)} \leq \frac{1}{c_{\phi}} \sum\limits_{d(y_i,y)\leq\frac{\delta}{2}} \vert \hat{a}_i(y) \vert^2 \leq \frac{1}{c_{\phi}} \left( \sum\limits_{d(y_i,y)\leq\frac{\delta}{2}} \vert \hat{a}_i(y) \vert \right)^2 \leq \frac{1}{c_{\phi}} C_2^2 \eqqcolon C(\phi) , \] 
where $C_2$ comes from property \ref{item:coeffs prop 3} of the coefficients $\hat{a}_i(y)$.
\end{proof}

Let us now consider the second factor in \eqref{eq:CSU}, which is the sum of all
(positive) weights $w(y_i,y)$ with respect to the center $y$. For a weight
function of the form \eqref{eq:weight function} we get:
\begin{equation}
  \label{eq:sum of weights}
  \sum\limits_{w(y_i,y)>0} w(y_i,y) = \sum\limits_{w(y_i,y)>0} \phi \left( \frac{d(y_i,y)}{\delta} \right)
  \leq \left\vert \left\{ i \in \{1,\dots,N\} \, \middle\vert \, d(y_i,y) < \delta \right\} \right\vert
  \cdot \max\limits_{t \in \lbrack 0,1 \rbrack} \phi(t) .
\end{equation}
Hence it only remains to show, that under the conditions of \Cref{lem:target
function bounded} the number of nodes $y_i$ that are contained in the spherical
cap $C_{\delta}(y)$ is uniformly bounded for all $y \in \S^{d-1}$. The
following estimate results from geometrical arguments and it can be found in the
proof of \cite[Theorem 3]{Wendland.2001}.

\begin{lemma}\label{lem:bound for number of neighbors} Let
$y_1,\dots,y_N \in \S^{d-1}$ be nodes with separation distance $q$. For every
center $y \in \S^{d-1}$ the number of nodes $y_i$, that are contained in the
spherical cap $C_{\delta}(y)$, is bounded by
\begin{align}\label{eq:bound for number of
neighbors} \left\vert \left\{ i \in \{1,\dots,N\} \, \middle\vert \, y_i \in C_{\delta}(y) \right\} \right\vert \leq \left( \frac {q + \delta}{q} \right)^{d-1} \, \left( \frac{\pi}{2} \right)^{d-2} .
\end{align}
\end{lemma}

Let us take a closer look at the expression $\nicefrac{(q+\delta)}{q}$.
By \Cref{lem:target function bounded} the constant $\delta$ should be chosen as a sufficiently large multiple of the fill distance $h$, i.e.\ $\delta = R \, h$ for some $R \geq r$, thus
\begin{equation}
  \label{eq:why uniformity matters}
  \frac{q+\delta}{q} = \frac{q+R h}{q} = 1 + R \, \frac{h}{q} .
\end{equation}
The ratio $\nicefrac{h}{q}$ of the fill distance and the separation distance is
essentially the ratio between the radius of the biggest hole in the nodes and
the minimal distance between two nodes.
This can be seen as a measure for how uniformly the nodes are spread, hence we
call $\nicefrac{h}{q}$ the uniformity of the nodes.

\begin{theorem}
  \label{thm:lebesgue constant bounded}
  Let $\phi \colon \lbrack 0,\infty ) \to \lbrack 0,\infty )$ be a continuous
  function satisfying \eqref{eq:weight fun prop 1}, \eqref{eq:weight fun
    prop 2} and let
  \begin{equation*}
    w(y,z) = \phi \left( \frac{d(y,z)}{\delta} \right)
  \end{equation*}
  be the corresponding weight function.
  For every subspace $\sG \subseteq \sH_L(\S^{d-1})$
  and every $c>0$ there exist constants $h_0,r,C(\phi,c)>0$, such that for
  arbitrary node sets $Y = \left\{ y_1,\dots,y_N \right\}$ with fill distance
  $h \leq h_0$ and uniformity $\nicefrac{h}{q} \leq c$ and
  $\delta = R \, h$ with $R \geq r$ the Lebesgue constant is uniformly bounded
  by $C(\phi,c)$
  \begin{equation*}
    \sL \left( \sG,w,Y,y \right) \leq C(\phi,c) , \quad y \in \S^{d-1}.
  \end{equation*}
\end{theorem} 
\begin{proof}
  This follows directly from combining \eqref{eq:CSU} with \Cref{lem:target
    function bounded}, \eqref{eq:sum of weights} and \Cref{lem:bound for number
    of neighbors} and inserting \eqref{eq:why uniformity matters}.
\end{proof}

We conclude this section by combining the uniform bound for the Lebesgue constant with the results of the previous section.
Recall, that we are interested in the decay of the error $\| f - \sM f \|_\infty$ of the MLS approximation as the number $N$ of nodes goes to infinity. 
Using a packing argument one can easily verify, that for $N \to \infty$ the separation distance $q$ must go to zero. 
If additionally the uniformity $\nicefrac{h}{q}$ is bounded, the fill distance must go to zero too.
This motivates the following corollary about the error of the MLS approximation
for $h \to 0$.

\begin{corollary}
  \label{cor:approx order of MLS} 
  Let $f \in \sC^{L+1}(\S^{d-1})$ and let
  $\left( Y^\ell = \{ y_1^\ell , \dots , y_{N_\ell}^\ell \} \right)_{\ell \in
    \N}$ be a family of node-sets with bounded uniformity
  $\nicefrac{h_\ell}{q_\ell} \leq c , \, \ell \in \N$, such that the fill
  distance $h_\ell$ goes to zero for $\ell \to \infty$. Further let
  $\phi \colon \lbrack 0,\infty) \to \lbrack 0,\infty)$ be a continuous
  function satisfying \eqref{eq:weight fun prop 1}, \eqref{eq:weight fun
    prop 2} and let
  \[ w_\ell (y,z) = \phi \left( \frac{d(y,z)}{\delta_\ell} \right) \]
  be the corresponding weight function.
  We denote the MLS approximation of $f$ with
  respect to the ansatz space $\sH_{L,2}(\S^{d-1})$, the nodes $Y^\ell$ and the
  weight function $w_\ell$ by
  $f^\ell \coloneqq \sM f$. Let $r>0$ be the constant from \Cref{lem:target
    function bounded}, choose $R \geq r$ and set
  $\delta_\ell \coloneqq R \, h_\ell , \, \ell \in \N$. The MLS approximation
  possesses approximation order $\sO \left( h^{L+1} \right)$ for
  $\ell \to \infty$ and $h_\ell \to 0$, i.e.\
  \[ \| f-f^\ell \|_\infty = \sO \left( {h_\ell}^{L+1} \right) . \]
\end{corollary} 

\begin{proof}
  First we recognize, that all conditions of \Cref{thm:lebesgue constant bounded}
  are satisfied. Since $h_\ell$ goes to zero, we may assume that the fill
  distance is smaller than the constant $h_0$ and the requirements with respect
  to the support radius are satisfied too. Thus the
  Lebesgue constant is uniformly bounded by $C(\phi,c)$, i.e.\
  \[ \sL \left( \sH_{L,2}(\S^{d-1}),w_{\ell},Y^\ell,y \right) \leq C(\phi,c) . \]
  Now $\delta_\ell \to 0$ for $\ell \to \infty$ and by \Cref{thm:main result} the error $E_{\delta_\ell} \left( \sH_{L,2}(\S^{d-1}),f,y \right)$ of the local best approximation of $f$ from $\sH_{L,2}(\S^{d-1})$ on the spherical cap $C_{\delta_\ell}(y)$ vanishes with order $\sO \left( \delta_\ell^{L+1} \right)$ for all functions $f \in \sC^{L+1}(\S^{d-1})$ and for all centers $y \in \S^{d-1}$.
  With the error estimate \eqref{eq:error est of MLS} of the MLS approximation we
  finally obtain for all $y \in \S^{d-1}$:
  \begin{align*}
    \vert f(y) - f^{\ell}(y) \vert &\leq E_{\delta_\ell} \left( \sH_{L,2}(\S^{d-1}),f,y \right) \cdot \left( 1 + \sL \left( \sH_{L,2}(\S^{d-1}),w_{\ell},Y^{\ell},y \right) \right) \\
   &\leq E_{\delta_\ell} \left( \sH_{L,2}(\S^{d-1}),f,y \right) \cdot \left( 1 + C(\phi,c) \right) = \left( 1 + C(\phi,c) \right) \cdot \sO \left( \delta_\ell^{L+1} \right) \\
   &= \left( 1 + C(\phi,c) \right) \cdot \sO \left( (R \, h_\ell)^{L+1} \right) = \left( 1 + C(\phi,c) \right) \cdot R^{L+1} \cdot \sO \left( {h_\ell}^{L+1} \right) .
  \end{align*}
\end{proof} 

 
\section{MLS Approximation with Polynomials on the Tangent Space of the
  Sphere}
\label{sec:tangentSpace} 

The most natural way to perform MLS approximation on a manifold is to locally
transfer the problem to the tangent space and use the ansatz space of
multivariate polynomials.
\edited{In contrast to the previous sections the ansatz space is  now
  depending on the evaluation point and on the local projection which is being
  used.

  The approach in this section was already discussed in \cite[Section
  3.1]{Sober.} in the setting of manifold approximation from scattered data,
  i.e.\ the manifold and in particular its tangent space were not given
  explicitly, but instead determined during the approximation process. Theorem
  3.2 from Sober et al.\ is directly applicable to our setting and proves
  approximation order $\sO(h^{L+1})$ if both the manifold and the function
  possess smoothness order $L+1$. In order to keep the assumptions in line with
  the assumptions in the previous sections, we give in the following a separate
  proof for this result in the spherical setting.}

To this end, let us introduce the MLS approximation with polynomials on the
tangent space of the sphere. For now we assume that the center is the north
pole, i.e.\ $y = e_d \in \S^{d-1}$, thus the tangent space is given by
\begin{equation*}
  T_{e_{d}} \S^{d-1} = \left\{ y \in \R^d \, \middle\vert \, y_d = 0 \right\} \cong \R^{d-1} .
\end{equation*}
Further, recall the inverse projection \eqref{eq:inverse projection of
  points} from the northern hemisphere to the ball
\[ \mapping{\inv{\pi}}{C_{\frac{\pi}{2}}(e_d)}{B_1(\vc{0})} , \quad
  \inv{\pi}(y_{1},\dots,y_{d}) = (y_1,\dots,y_{d-1}) . \]
Assume that $w \colon \S^{d-1} \times \S^{d-1} \to \R$ is a weight function and
$y_{1},\dots,y_{N} \in \S^{d-1}$ are nodes, such that
\begin{equation*}
  \min\limits_{p \in \sP_L(\R^{d-1})} \sum\limits_{w(y_i,e_d)>0}
  w(y_i,e_d) \, \vert f(y_i) - \concat{p}{\inv{\pi}} (y_i) \vert^2
\end{equation*}
possesses a unique minimizer $p_{e_{d}} \in \sP_{L}(\R^{d-1})$. Then analogous
to \Cref{def:def mls} we may define the MLS approximation with respect to
polynomials on the tangent space as $\sM f(y) \coloneqq p_{e_{d}}(\vc{0})$.
This can also be interpreted as the MLS approximation with respect to the ansatz
space
\begin{align}\label{eq:tangent ansatz space
  e_d} \concat{\sP_L(\R^{d-1})}{\inv{\pi}} \coloneqq \left\{ \concat{p}{\inv{\pi}} \, \middle\vert \, p \in \sP_L(\R^{d-1}) \right\} ,
\end{align} which is just the composition of all polynomials with $\inv{\pi}$.
Note, that this ansatz space is dependent on the specific choice of the local
projection.

We want to generalize the above approach to arbitrary $y \in \S^{d-1}$. To this
end we rotate the center $y \in \S^{d-1}$ and its neighbors into the north pole
and proceed in the same manner as before.

\begin{definition}\label{def:tangent mls}
  Let a function $f \colon \S^{d-1} \to \C$ and a center $y \in \S^{d-1}$ be
  given and let $\mt{R}_y \in \SO{d}$ be a rotation matrix, such that $y$ gets
  mapped to the north pole, i.e.\ $\mt{R}_y y = e_d$. Further, let
  $y_{1},\dots,y_{N} \in \S^{d-1}$ be nodes and let
  $w \colon \S^{d-1} \times \S^{d-1} \to \R_{\geq 0}$ be a weight function, such
  that the minimization problem
  \begin{align}\label{eq:mls tangent
    space} p_y = \argmin\limits_{p \in \sP_L(\R^{d-1})} \sum\limits_{w(y_i,y)>0} w(y_i,y) \, \vert f(y_i) - \concatt{p}{\inv{\pi}}{\mt{R}_y} (y_i) \vert^2
  \end{align}
  possesses a unique solution $p_{y} \in \sP(\R^{d-1})$.
  Then $\sM f(y) \coloneqq p_y(\vc{0})$ is the MLS approximation of $f$ in $y$
  with respect to polynomials on the tangent space.
\end{definition}

This can be seen as MLS approximation on the sphere with respect to the ansatz space
\begin{align}\label{eq:tangent ansatz
  space} \widetilde{\sP}_L(\R^{d-1},\pi,y) \coloneqq \concatt{\sP_L(\R^{d-1})}{\inv{\pi}}{\mt{R}_y} \coloneqq \left\{ \concatt{p}{\inv{\pi}}{\mt{R}_y} \, \middle\vert \, p \in \sP_L(\R^{d-1}) \right\} ,
\end{align} which now depends on the center $y \in \S^{d-1}$ and on the local
projection $\pi$. Note,
that all following arguments regarding the space
$\widetilde{\sP}_L(\R^{d-1},\pi,y)$ do not depend on the specific matrix
$\mt{R}_y \in \SO{d}$, but only on the property $\mt{R}_y \, y = e_d$.
As discussed after \Cref{def:def mls} the minimization problem \eqref{eq:mls
  tangent space} possesses a unique solution, if and only if the nodes
$\left\{ y_{i} \, \middle\vert \, w(y_{i},y) > 0 \right\}$ with positive weight
are unisolvent with respect to the ansatz space \eqref{eq:tangent ansatz
  space}.

We want to show, that the MLS approximation from \Cref{def:tangent mls} also
attains the approximation order $\sO \left( h^{L+1} \right)$ if the conditions
of \Cref{cor:approx order of MLS} are satisfied. By \Cref{lem:MLS error
  estimate} it is sufficient to check, that the corresponding  ansatz
space \eqref{eq:tangent ansatz space} has local approximation order
$\sO \left( \delta^{L+1} \right)$ for $\delta \to 0$ and that the related
Lebesgue constant can be uniformly bounded. We validate both statements in the
following two lemmata.

\begin{lemma}\label{lem:tangent local approx order L+1} For all functions
  $f \in \sC^{L+1}(\S^{d-1})$ and for all centers $y \in \S^{d-1}$ the error of
  the local best approximation of $f$ on the spherical cap $C_{\delta}(y)$ with
  respect to the ansatz space $\widetilde{\sP}_L(\R^{d-1},\pi,y)$ from
  \eqref{eq:tangent ansatz space} decays with order $L+1$ as $\delta$ goes to
  zero:
  \[ E_{\delta}(\widetilde{\sP}_L(\R^{d-1},\pi,y),f,y) = \sO(\delta^{L+1}) . \]
\end{lemma}
\begin{proof} Let us first consider the case $y = e_d$, where the ansatz space
  is given by
  \begin{equation}
    \label{eq:rotated polynomials wrt north pole}
    \widetilde{\sP}_L(\R^{d-1},\pi,e_{d}) = \left\{ p \circ \inv{\pi} \circ \id \; \middle\vert \; p \in \sP (\R^{d-1}) \right\} = \left\{ p \circ \inv{\pi} \; \middle\vert \; p \in \sP (\R^{d-1}) \right\} .
  \end{equation}
  The pullback function of one element $\concat{p}{\inv{\pi}}$ is
  $\left( \concat{p}{\inv{\pi}} \right)^{\ast} = \concatt{p}{\inv{\pi}}{\pi} = p$
  and hence the truncated Taylor series of order $L$ of the pullback of
  $\concat{p}{\inv{\pi}}$ simplifies to
  \[ T_L( \concat{p}{\inv{\pi}} ) = \sum\limits_{\vc{\alpha} \in \N_0^{d-1} , \, \vert \vc{\alpha} \vert \leq L} \frac{ \left( (\concat{p}{\inv{\pi}})^{\ast} \right)^{(\vc{\alpha})}}{\vc{\alpha}!} \vc{x}^{\vc{\alpha}} = \sum\limits_{\vc{\alpha} \in \N_0^{d-1} , \, \vert \vc{\alpha} \vert \leq L} \frac{ p^{(\vc{\alpha})}}{\vc{\alpha}!} \vc{x}^{\vc{\alpha}} = p . \]
  As a direct consequence we obtain
  $T_L \left( \concat{\sP_L(\R^{d-1})}{\inv{\pi}} \right) = \sP_L(\R^{d-1})$ and
  by the proof of \Cref{lem:Phi(G) = all polynomials up to
    degree L} this verifies the desired approximation order for the north pole for
  $\delta \to 0$:
  \[ E_{\delta}(\concat{\sP_L(\R^{d-1})}{\inv{\pi}},f,e_d) = \sO \left( \delta^{L+1} \right) , \quad f \in \sC^{L+1}(\S^{d-1}) . \]

  Now assume $y \neq e_d$ and let $\mt{R}_y \in \SO{d}$ be a rotation matrix with
  $\mt{R}_y y = e_d$, then we obtain:
  \begin{align*} E_{\delta} \left( \widetilde{\sP}_L(\R^{d-1},\pi,y),f,y \right) &= \inf\limits_{p \in \sP_L(\R^{d-1})} \sup\limits_{z \in C_{\delta}(y)} \vert f(z) - p \circ \pi^{-1} \circ \mt{R}_y(z) \vert \\ &= \inf\limits_{p \in \sP_L(\R^{d-1})} \sup\limits_{z \in C_{\delta}(e_d)} \vert \concat{f}{\inv{\mt{R}}_y}(z) - p \circ \pi^{-1}(z) \vert \\ &= E_{\delta} \left( \concat{\sP_L(\R^{d-1})}{\inv{\pi}}, f \circ \mt{R}^{-1}_y,e_d \right)
  \end{align*} and by the first part of the proof the last
  term decays with order $\sO( \delta^{L+1} )$ for $\delta \to 0$ since
  $f \circ \mt{R}_y^{-1} \in \sC^{L+1}(\S^{d-1})$.
\end{proof}

It remains to give a uniform bound for the Lebesgue constant. By
\Cref{thm:lebesgue constant bounded} it is sufficient to show, that the ansatz
space $\widetilde{\sP}_L(\R^{d-1},\pi,y)$ is contained in the space
$\sH_L(\S^{d-1})$ of spherical harmonics up to degree $L$.
Since a function $\widetilde{p} \in \widetilde{\sP}_L(\R^{d-1},\pi,y)$ is not
defined on the whole sphere, but only on the spherical cap $C_{\frac \pi 2}(y)$,
we must instead show a slightly different statement, namely that $\widetilde{p}$
is the restriction of a spherical harmonic $h \in \sH_{L}(\S^{d-1})$ to the
spherical cap $C_{\frac \pi 2}(y)$.

\begin{lemma}\label{lem:tangent lebesgue constant bounded} Let
  $L \in \N_0$ and $y \in \S^{d-1}$ be given.
  Every function from the ansatz space $\widetilde{\sP}_L(\R^{d-1},\pi,y)$ is the restriction
  of a spherical harmonic of degree $\leq L$ to the domain of this function,
  i.e.\ for every $\widetilde{p} \in \widetilde{\sP}_L(\R^{d-1},\pi,y)$ there
  exists $h \in \sH_{L}(\S^{d-1})$, such that
  \[ \widetilde{p} (y) = h(y) , \quad y \in C_{\frac \pi 2}(y) . \]
\end{lemma}
\begin{proof} First assume $y=e_d$. By \eqref{eq:rotated polynomials wrt north
    pole} every function
  $\widetilde{p} \in \widetilde{\sP}_L(\R^{d-1},\pi,e_{d})$ possesses the form
  $\widetilde{p} = p \circ \inv{\pi} , \ p \in  \sP ( \R^{d-1} )$. For
  $y \in C_{\frac \pi 2} (e_{d})$ the function $\widetilde{p}$ can be written as
  \begin{align}\label{eq:projected polynomial is
    polynomial} p \circ \inv{\pi} (y) = p(y_1,\dots,y_{d-1}) = \sum\limits_{\widetilde{\vc{\alpha}} \in \N_0^{d-1} , \, \vert \widetilde{\vc{\alpha}} \vert \leq L} c_{\alpha} (y_1,\dots,y_{d-1})^{\widetilde{\vc{\alpha}}} = \sum\limits_{\vc{\alpha} \in \N_0^d , \, \vert \alpha \vert \leq L , \, \alpha_d=0} c_{\alpha} y^{\vc{\alpha}} .
  \end{align}
  Hence $\widetilde{p}$ is the restriction of a spherical harmonic
  by \Cref{lem:harmonic space contains all spherical polynomials}.

  For $y \neq e_{d}$ the ansatz space
  $\widetilde{\sP}_L(\R^{d-1},\pi,y)$ contains the same
  functions as $\widetilde{\sP}_L(\R^{d-1},\pi,e_{d})$, just
  preceded by some rotation matrix $\mt{R}_{y} \in \SO{d}$ with
  $\mt{R}_{y} y = e_{d}$, see \eqref{eq:tangent ansatz space}.
  Because every harmonic space is rotational invariant this does not affect,
  that $\widetilde{p} \in  \widetilde{\sP}_L(\R^{d-1},\pi,y)$ may
  represented by spherical harmonics on its domain.
\end{proof}

Hence all conditions of \Cref{cor:approx order of MLS} with respect to the
ansatz space are satisfied and we can transfer the result to the MLS
approximation with respect to polynomials on the tangent space.

\begin{corollary}
  \label{cor:tanget_space} The MLS approximation of $f \in \sC^{L+1}(\S^{d-1})$
  in $y \in \S^{d-1}$ with respect to the ansatz space
  \[ \widetilde{\sP}_L(\R^{d-1},\pi,y) = \left\{ p \circ \inv{\pi} \circ \mt{R}_y \, \middle\vert \, p \in \sP(\R^{d-1}) , \mt{R}_y \in \SO{d} , \, \mt{R}_y y = e_d \right\} , \]
  from \eqref{eq:tangent ansatz space} possesses the approximation order
  $\sO \left( h^{L+1} \right)$ for $h \to 0$ under the same constraints as for
  the ansatz space $\sH_{L,2}(\S^{d-1})$, see \Cref{cor:approx order of MLS}.
\end{corollary}

 
\section{Numerical Comparison}
\label{sec:numerics}

In this section we perform MLS approximation on the $2$-sphere $\S^2$ with
respect to the 3 ansatz spaces of the previous sections: the even,
respectively odd spherical harmonics up to degree $L \in \N$, see
\eqref{eq:even/odd spherical harmonics}, all spherical harmonics up to degree
$L$, and polynomials on the tangent space up to degree $L$, see
\eqref{eq:tangent ansatz space}. The goal is to verify the theoretical
approximation order $\sO \left(h^{L+1}\right)$ for $\delta \to 0$ for all 3
ansatz spaces and analyze the numerical behavior.

We will only present the results for polynomial degree $L=3$. The numerical
experiments for $L=1,2,4$ showed similar characteristics. Higher polynomial
degrees are not well suited for MLS approximation, since the goal is the fast
computation of many local approximations of low degree. For $L=0$ all three
ansatz spaces are identical.

Let us first consider only the spaces $\sH_{3,2}(\S^{2})$ and
$\sH_{3}(\S^{2})$. Let
\begin{equation*}
  Y_{\ell} \coloneqq \left\{ Y_{\ell}^{m} \, \middle\vert \, m = -\ell, \dots, \ell \right\}
\end{equation*}
be an $L^{2}$-orthonormal basis of $\Harm_{\ell}(\S^{2})$, see \cite[Section 5.2]{VMK}.
We utilize the bases
\begin{align*}
  Y_{L} \coloneqq \bigcup\limits_{\ell=0}^{L} Y_{\ell} \, , \qquad \qquad
  Y_{L,2} \coloneqq \bigcup\limits_{\ell \leq L, \, \ell = L \pmod*{2}}^{L} Y_{\ell}
\end{align*}
for $\sH_{L}(\S^{2})$ and $\sH_{L,2}(\S^{2})$, respectively.

By combining \eqref{eq:minimal dimension} and \eqref{eq:dim of sum of poly
  spaces} we obtain $\dim \left( \sH_{3,2} (\S^{2}) \right) = 10$ and from the
decomposition $\sH_{3}(\S^{2}) = \sH_{3,2}(\S^{2}) \oplus \sH_{2,2}(\S^{2})$ we
deduce $\dim \left( \sH_{3}(\S^{2}) \right) = 10 + 6 = 16$. Further we utilize
the weight function
\begin{equation*}
  w \colon \S^2 \times \S^2 \to \lbrack 0,1 \rbrack , \quad w(y,z)
  := \left( \max \left\{ 1 - \frac{d(y,z)}{\delta} , 0   \right\} \right)^{2} ,
\end{equation*}
which satisfies the requirements \eqref{eq:weight function}, \eqref{eq:weight
  fun prop 1} and \eqref{eq:weight fun prop 2}. Regarding \Cref{thm:lebesgue
  constant bounded} we have to choose $\delta = R \, h$ as a sufficiently
large multiple of the fill distance $h$ of the nodes in order to guarantee the
existence of the MLS approximation. On the other hand, from
the proof of \Cref{cor:approx order of MLS} we can read off that the
approximation error grows with $R \to \infty$. For $\sH_{3,2}(\S^2)$ a
reasonable factor is $R=3.5$. As $\sH_3(\S^2)$ has a higher dimension we need
to choose here $R=4.5$.

We use the test function $f \colon \S^{2} \to \R$ from \cite{KurtJetter.1999,
  Wendland.2001} which is a superposition of $5$ exponentials
\begin{equation*}
  f(y) = \sum\limits_{i=1}^{5} f_{i}(y) = \sum\limits_{i=1}^{5} c_{i} \cdot
  \exp \left( {-\alpha_{i} \left( 1 - \langle p_{i},y \rangle \right)} \right) .
\end{equation*}
The parameters can be found in \Cref{tab:table of parameters}. Further a
surface plot of $f$ is provided in \Cref{fig:testfun}.

\begin{figure}[t!]
  \centering
  \begin{minipage}[t]{\textwidth}
    \begin{minipage}[c]{.33\textwidth}
      \centering
      \includegraphics[width=\textwidth]{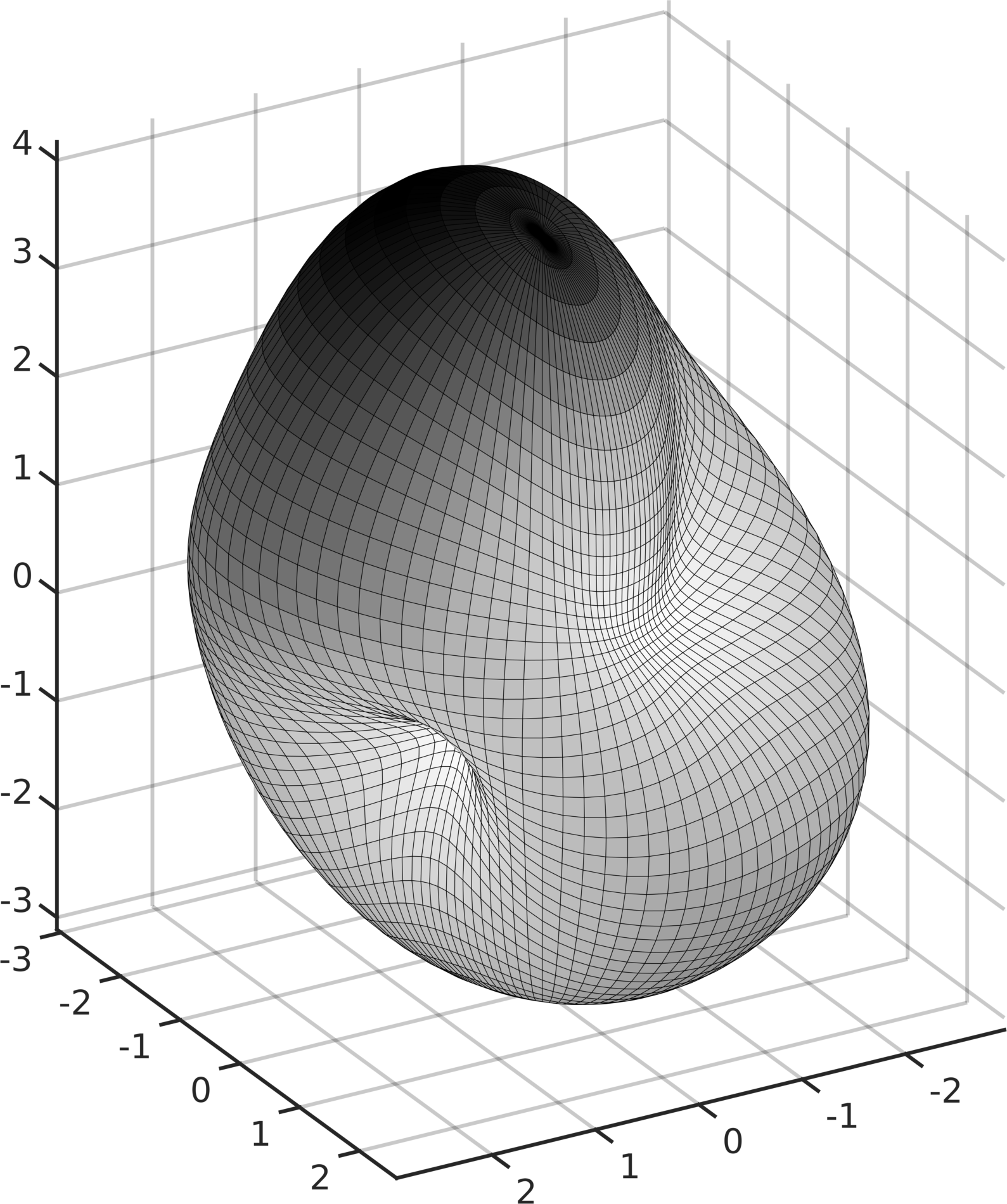}
      \captionof{figure}{Surface Plot of \\ \hphantom{Figure 1:} $(3+f(y)) \cdot y$}
      \label{fig:testfun}
    \end{minipage}
    \hfill
    \begin{minipage}[c]{.66\textwidth}
      \centering
      \captionof{table}{Parameters defining the test function $f$}
      \begin{tabular}[c]{rrrrrrr}
        \label{tab:table of parameters}
        $i$ & $p_{i,1}$ & $p_{i,2}$ & $p_{i,3}$ & $n_{i}$ & $\alpha_{i}$ & $c_{i}$ \\ \midrule
        1 & 0 & 0 & 1 & 1 & 5 & 2 \\
        2 & 0.932039 & 0 & 0.362358 & 1 & 7 & 0.5 \\
        3 & -0.362154 & 0.612280 & 0.696707 & 2 & 6 & -2 \\
        4 & 0.904035 & 0.279651 & -0.323290 & 1 & 5 & -2 \\
        5 & -0.047932 & -0.424684 & -0.904072 & 1 & 2.1 & 0.2
      \end{tabular}
    \end{minipage}
  \end{minipage}
\end{figure}

\begin{figure}[b!]
  \captionsetup{type=table}
  \centering
  \caption{The fill distance $h$ and the separation distance $q$ in degree and
    the resulting uniformity $\nicefrac hq$ for Fibonacci grids with $N = 2n+1$
    nodes, $n \in \{ 5 \cdot 2^{1}, 5 \cdot 2^{2}, \dots, 5 \cdot 2^{26} \}$.}
  \begin{filecontents*}{data/grid.txt}
        size         fill       sep       hom
          21      33.8179   19.6790    1.7185
          41      24.2987   13.9551    1.7412
          81      17.2516    9.8848    1.7453
         161      12.0852    6.9961    1.7274
         321       8.6072    4.9494    1.7390
         641       6.0951    3.5006    1.7411
        1281       4.0756    2.4756    1.6463
        2561       2.9744    1.7506    1.6990
        5121       2.1485    1.2379    1.7356
       10241       1.4000    0.8754    1.5994
       20481       1.0630    0.6190    1.7174
       40961       0.7056    0.4377    1.6120
       81921       0.4956    0.3095    1.6013
      163841       0.3539    0.2188    1.6170
      327681       0.2461    0.1547    1.5903
      655361       0.1746    0.1094    1.5954
     1310721       0.1241    0.0774    1.6040
     2621441       0.0866    0.0547    1.5834
     5242881       0.0620    0.0387    1.6026
    10485761       0.0433    0.0274    1.5837
    20971521       0.0312    0.0193    1.6131
    41943041       0.0218    0.0137    1.5933
    83886081       0.0154    0.0097    1.5927
   167772161       0.0109    0.0068    1.6007
   335544321       0.0078    0.0048    1.6119
   671088641       0.0055    0.0034    1.5995
\end{filecontents*}
\pgfplotstabletypeset[
  columns={size, fill, sep, hom, size, fill, sep, hom},
  display columns/0/.style={select equal part entry of={0}{2}, column name=$N$,
    fixed, dec sep align},
  display columns/1/.style={select equal part entry of={0}{2}, column name=$h$,
    precision=4, fixed, zerofill={true}, dec sep align},
  display columns/2/.style={select equal part entry of={0}{2}, column name=$q$,
    precision=4, fixed, zerofill={true}, dec sep align},
  display columns/3/.style={select equal part entry of={0}{2}, column
    name=$\nicefrac hq$, precision=2, fixed, zerofill={true}, dec sep align},
  display columns/4/.style={select equal part entry of={1}{2}, column name=$N$,
    fixed, dec sep align},
  display columns/5/.style={select equal part entry of={1}{2}, column name=$h$,
    precision=4, fixed, zerofill={true}, dec sep align},
  display columns/6/.style={select equal part entry of={1}{2}, column name=$q$,
    precision=4, fixed, zerofill={true}, dec sep align},
  display columns/7/.style={select equal part entry of={1}{2}, column
    name=$\nicefrac hq$, precision=2, fixed, zerofill={true}, dec sep align}
  ]
  {data/grid.txt}
  \label{tab:grid_properties}
\end{figure}

It remains to construct node-sets with decreasing fill distance $h$ and
bounded uniformity $\nicefrac{h}{q}$ in order to verify the approximation
order $\sO \left( h^4 \right)$ from \Cref{cor:approx order of MLS}. Here we
utilize Fibonacci grids \cite{shepard68} with $2n+1$ nodes,
$n \in \left\{ 5 \cdot 2^{1}, 5 \cdot 2^{1}, \dots, 5 \cdot 2^{26} \right\}$.
The characteristics of these grids can be found in \Cref{tab:grid_properties}.

In order to estimate the error of the MLS approximation we use a discrete 
approximation of the supremum norm
\[ L_{\infty} := \max\limits_{t_i \in T} \vert f(t_i) - \sM f(t_i) \vert
  \approx \sup\limits_{y \in \S^2} \vert f(y) - \sM f(y) \vert , \] where
$T \subset \S^2$ is a set of uniformly distributed random points of
cardinality $\vert T \vert = 10^5$. These error estimates are depicted by the
dashed lines in \Cref{fig:error plots} for the bases $Y_{3}$ and $Y_{3,2}$
and decreasing fill distances.
\edited{

}

For $h>3$ degrees in the case of $Y_{3}$
and for $h>0.2$ degrees in the case of $Y_{3,2}$ the theoretical
approximation order $\sO({h}^4)$ is attained very precisely. For $h<3$ degrees
the basis $Y_{3}$ becomes locally linearly depended which results in a very
badly conditioned Gram matrix, see \Cref{fig:conds}. Since the space
$\mathcal H_{L,2}(\S^{2})$ is much better adopted to local approximation its
basis $Y_{3,2}$ remains stable for smaller values of $h$.
\edited{
  But still, with
  decreasing fill distance the condition of the Gram matrix becomes worse, see
  \Cref{fig:conds}, and numerical instability occurs at $h \approx 0.2$ degrees.
}

\edited{
  \begin{remark}
    If an evaluation point is close to a zero of a basis function, all entries
    of the corresponding row and column of the Gram matrix are close to zero
    too. We thus rescaled the basis functions in all numerical experiments, such
    that the absolute values of the diagonal entries of the Gram matrix are $1$.
  \end{remark}
}

\begin{figure}[t]
  \centering
  \label{fig:conds and err plots}
  \begin{subfigure}[t]{.48\textwidth}
    \centering
    \begin{tikzpicture}
      \begin{axis}[ enlargelimits=false, scale=.95,
        grid = major,
        tick align = outside, xmode = log, ymode = log, xlabel = $h$,
        ymin=1, ymax = 10^22,
        xtick pos = bottom, ytick pos = left, 
        xticklabel={$\pgfmathparse{exp(\tick)}\pgfmathprintnumber[fixed]{\pgfmathresult}^{\circ}$},
        yticklabel style = {rotate=90},
        legend cell align = left,
        legend style={nodes={scale=0.85, transform shape},at={(0,0.3)},anchor = west}]
        \pgfplotstableread{data/conds.txt}\mydata;
        \addplot+[no marks, line join = round,color=blue,style=dashed,thick] table
        [x=filldist, y expr={exp(\thisrow{all_harm})}] {\mydata};
        \addlegendentry{$Y_{3}$}
        \addplot+[no marks, line join = round, color=blue, opacity=0.7, thick] table
        [x=filldist, y expr={exp(\thisrow{even_harm})}] {\mydata};
        \addlegendentry{$Y_{3,2}$}
        \addplot+[no marks, line join = round, color=red, opacity=0.7, thick] table
        [x=filldist, y expr={exp(\thisrow{even_mon_cent})}] {\mydata};
        \addlegendentry{$\sB_{3,2}(y)$}
        \addplot+[no marks, line join = round, color=black, style=dashed, thick] table
        [x=filldist, y expr={exp(\thisrow{tangent})}] {\mydata};
        \addlegendentry{$\widetilde{\sP}_3(\R^{d-1},\inv{\pi},y)$}
      \end{axis} \end{tikzpicture}
    \newsubcap{The worst condition number of the Gram matrices that
      occurred among all evaluation points, for the $4$ different ansatz spaces
      as the fill distance goes to zero.}
    \label{fig:conds}
  \end{subfigure}
  \hfill
  \begin{subfigure}[t]{.48\textwidth}
    \centering
    \begin{tikzpicture}
      \begin{axis}[enlargelimits=false, scale=.95,
        grid = major,
        tick align = outside, xmode = log, ymode = log, xlabel = $h$,
        xtick pos = bottom, ytick pos = left, 
        xticklabel={$\pgfmathparse{exp(\tick)}\pgfmathprintnumber[fixed]{\pgfmathresult}^{\circ}$},
        yticklabel style = {rotate=90},
        legend pos = south east, legend cell align = left, legend
        style={nodes={scale=0.85, transform shape}}]
        \pgfplotstableread{data/errors.txt}\mydata;
        \addplot+[no marks, line join = round,color=blue,style=dashed,thick] table
        [x=filldist, y=all_harm] {\mydata};
        \addlegendentry{$Y_{3}$}
        \addplot+[no marks, line join = round, color=blue, opacity=0.7, thick] table
        [x=filldist, y=even_harm] {\mydata};
        \addlegendentry{$Y_{3,2}$}
        \addplot+[no marks, line join = round, color=red, opacity=0.7, thick] table
        [x=filldist, y=even_mon_cent] {\mydata};
        \addlegendentry{$\sB_{3,2}(y)$}
        \addplot+[no marks, line join = round, color=black, style=dashed, thick] table
        [x=filldist, y=tangent] {\mydata};
        \addlegendentry{$\widetilde{\sP}_3(\R^{d-1},\inv{\pi},y)$}
      \end{axis} \end{tikzpicture}
    \newsubcap{Decay of the $L_{\infty}$ error \eqref{eq:error est of MLS} of the
      MLS approximation of $f$ w.r.t. different ansatz spaces and bases as the
      fill distance goes to zero.}
    \label{fig:error plots}
  \end{subfigure}
\end{figure}

An easy and effective way to obtain an approximation process which remains
stable for even smaller fill distances $h$ is to replace the basis $Y_{3,2}$ of
$\mathcal H_{L,2}(\S^{2})$ by a locally adapted basis of the same space. To this
end we consider the monomial basis $\sB_{3,2}$ from \Cref{lem:basis of G_L}
and utilize for $y \in \S^2$ an arbitrary rotation $\mt{R}_y \in \SO{3}$ that
maps $y$ into the north pole, i.e.\ $\mt{R}_y y = e_3$, in order to define a
local basis
\begin{equation*}
  \sB_{3,2}(y) \coloneqq \left\{ g \circ \mt{R}_y \, \middle\vert \, g \in \sB_{3,2} \right\}
\end{equation*}
of the same space. The corresponding approximation error
$L_\infty$ is depicted by the red line in \Cref{fig:error plots} and decays
all the way down to almost machine accuracy.
\edited{
  Additionally, the condition of the Gram matrices remains constant
  even if the fill distance becomes very small, as depicted in \Cref{fig:conds}.
}

A similar approximation error is obtained when using the local ansatz spaces
$\widetilde{\sP}_3(\R^2,\inv{\pi},y)$ as introduced in
\cref{sec:tangentSpace}. Here we consider the monomial basis \eqref{eq:monomial
  basis} on the tangent space $T_{e_{3}} \S^{2}$ to the sphere at the north
pole $e_{3}$ and use the rotation $\mt{R}_{y}$ and the projection
$\pi \colon T_{e_3} \S^2 \to \S^2$ to derive local ansatz spaces for each $y
\in \S^{2}$. The most significant difference to the previous setting
is that not only the basis changes with the local evaluation point but
the entire ansatz space.
\edited{
  As for the basis $\sB_{3,2}(y)$, the conditions
  of the Gram matrices are not affected by the fill distance.
}

\ \newline\noindent\textbf{Data Availability }
All numerical experiments have been realized using the Matlab Toolbox MTEX. The
corresponding scripts are available
at \url{https://github.com/mtex-toolbox/mtex-paper/tree/master/MLSonSpheres}.



\begin{thebibliography}{10}

\bibitem{AtHa12}
K.~Atkinson and W.~Han.
\newblock {\em Spherical Harmonics and Approximations on the Unit Sphere: An
  Introduction}, volume 2044 of {\em Lecture Notes in Mathematics}.
\newblock Springer, Heidelberg, 2012.

\bibitem{FiTh06}
F.~Filbir and W.~Themistoclakis.
\newblock Polynomial approximation on the sphere using scattered data.
\newblock {\em Math. Nachr.}, 281:650--668, 2008.

\bibitem{Ku07}
S.~Kunis.
\newblock A note on stability results for scattered data interpolation on
  {E}uclidean spheres.
\newblock {\em Adv. Comput. Math.}, 30:303--314, 2009.

\bibitem{filbir2023marcinkiewiczzygmund}
F.~Filbir, R.~Hielscher, T.~Jahn, and T.~Ullrich.
\newblock Marcinkiewicz--zygmund inequalities for scattered and random data on
  the $q$-sphere, 2023.

\bibitem{potts.1998}
D.~Potts, G.~Steidl, and M.~Tasche.
\newblock Fast and stable algorithms for discrete spherical fourier transform.
\newblock {\em Linear Algebra and its Applications}, (Nr.275):433--450, 1998.

\bibitem{potts.2008}
J.~Keiner and D.~Potts.
\newblock Fast evaluation of quadrature formulae on the sphere.
\newblock {\em Math. Comput.}, 77:397--419, 01 2008.

\bibitem{doublesphere}
S.~Mildenberger and M.~Quellmalz.
\newblock Approximation properties of the double fourier sphere method.
\newblock {\em Journal of Fourier Analysis and Applications}, 28, 04 2022.

\bibitem{doulesphereNdim}
S.~Mildenberger and M.~Quellmalz.
\newblock A double fourier sphere method for d-dimensional manifolds.
\newblock {\em Sampling Theory, Signal Processing, and Data Analysis}, 21, 07
  2023.

\bibitem{Golitschek.2001}
M.~v.~Golitschek and W.~A. Light.
\newblock Interpolation by polynomials and radial basis functions on spheres.
\newblock {\em Constructive Approximation}, 17(1):1--18, 2001.

\bibitem{KurtJetter.1999}
{K. Jetter}, {J. St{\"o}ckler}, and {J. D. Ward}.
\newblock Error estimates for scattered data interpolation on spheres.
\newblock {\em Math. Comput.}, 68:733--747, 1999.

\bibitem{rbflocal}
K.~Hesse and Q.~T.~Le Gia.
\newblock Local radial basis function approximation on the sphere.
\newblock {\em Bulletin of the Australian Mathematical Society}, 77:197 -- 224,
  04 2008.

\bibitem{Levin.1998}
D.~Levin.
\newblock The approximation power of moving least-squares.
\newblock {\em Mathematics of Computation}, 67(224):1517--1531, 1998.

\bibitem{Sober.}
Barak Sober, Yariv Aizenbud, and David Levin.
\newblock Approximation of functions over manifolds: A moving least-squares
  approach.
\newblock {\em Journal of Computational and Applied Mathematics}, 383:113140,
  2021.

\bibitem{Wendland.2001b}
H.~Wendland.
\newblock Local polynomial reproduction and moving least squares approximation.
\newblock {\em IMA Journal of Numerical Analysis}, 21(1):285--300, 2001.

\bibitem{Wendland.2001}
H.~Wendland.
\newblock Moving least squares approximation on the sphere.
\newblock {\em Mathematical Methods for Curves and Surfaces}, 2001.

\bibitem{Fasshauer.2008}
G.~E. Fasshauer.
\newblock {\em Meshfree approximation methods with MATLAB}, volume~6 of {\em
  Interdisciplinary mathematical sciences}.
\newblock {World Scientific}, New Jersey, reprint edition, 2008.

\bibitem{shepard68}
D.~Shepard.
\newblock A two-dimensional interpolation function for irregularly spaced
  points.
\newblock In {\em Proc. 1968 Assoc. Comput. Machinery National Conference},
  pages 517--524, 1968.

\bibitem{backus67}
G.~Backus and F.~Gilbert.
\newblock Numerical applications of a formalism for geophysical inverse
  problems.
\newblock {\em Geophys. J. R. Astron. Soc.}, 13:247--276, 1967.

\bibitem{bos89:_movin}
L.~Bos and K.~Salkauskas.
\newblock Moving least-squares are backus-gilbert optimal.
\newblock {\em J. Approx. Theory}, 59:267--275, 1989.

\bibitem{Efthimiou.2014}
C.~J. Efthimiou and C.~Frye.
\newblock {\em Spherical harmonics in p dimensions}.
\newblock {World Scientific Pub. Co}, Singapore and Hackensack, N.J, 2014.

\bibitem{Muller.1966}
C.~M{\"u}ller.
\newblock {\em Spherical Harmonics}, volume~17 of {\em Springer eBook
  Collection Mathematics and Statistics}.
\newblock {Springer Berlin Heidelberg}, Berlin, Heidelberg, 1966.

\bibitem{Axler.2001}
S.~J. Axler.
\newblock {\em Harmonic Function Theory}, volume 137 of {\em Springer eBook
  Collection Mathematics and Statistics}.
\newblock Springer, New York, NY, second edition edition, 2001.

\bibitem{VMK}
D.~A. Varshalovich, A.~N. Moskalev, and V.~K. Khersonskii.
\newblock {\em Quantum Theory of Angular Momentum}.
\newblock World Scientific, Singapore, 1988.

\end{thebibliography}

\cleardoublepage


\end{document}